\documentclass[a4paper,12pt]{amsart}
\usepackage{amsmath,amssymb,amsthm,amsxtra,stmaryrd} 
\usepackage{extarrows}
\usepackage{thmtools,color}
\usepackage{mathtools}
\usepackage{float} 
\usepackage{graphics}
\usepackage{caption}
\usepackage{enumerate}
\usepackage[dvipsnames]{xcolor}
\usepackage[normalem]{ulem}   

\usepackage{hyperref} %
\hypersetup{unicode,hypertexnames=false,colorlinks=true,linkcolor=blue,citecolor=blue,pdfauthor={Goldman-Merlet-Pegon-Serfaty},pdftitle={Compactness and structure of zero-states for unoriented Aviles--Giga functionals}}

\numberwithin{equation}{section}

\theoremstyle{theorem}
\newtheorem{theorem}{Theorem}[section]
\newtheorem{lemma}[theorem]{Lemma}
\newtheorem{proposition}[theorem]{Proposition}
\newtheorem{corollary}[theorem]{Corollary}
\newtheorem{definition}[theorem]{Definition}
\newtheorem{properties}[theorem]{Properties}
\newtheorem{remark}[theorem]{Remark}

\theoremstyle{theorem}
\newtheorem*{theorem*}{Theorem}
\newtheorem*{lemma*}{Lemma}
\newtheorem*{proposition*}{Proposition}
\newtheorem*{corollary*}{Corollary}
\newtheorem*{definition*}{Definition}
\newtheorem*{properties*}{Properties}
\newtheorem*{remark*}{Remark}
\newtheorem*{problem*}{Problem}
\newtheorem*{question*}{Question}
\newtheorem*{notation*}{Notation}

\declaretheoremstyle[
  headfont=\color{Blue}\normalfont\bfseries,
  bodyfont=\color{Blue}\normalfont
]{colored}

\newcommand{\com}[1]{}

\newcommand{\nono}{\widehat}
\newcommand{\les}{\lesssim}
\newcommand{\ges}{\gtrsim}

\newcommand{\R}{\mathbb{R}}

\let\C\relax 
\newcommand{\C}{\mathbb{C}}
\newcommand{\Z}{\mathbb{Z}}
\newcommand{\T}{\mathbb{T}}

\newcommand{\So}{\mathbb{S}^1}

\newcommand{\Ups}{\Upsilon}

\newcommand{\hel} {
\hskip2.5pt{\vrule height7pt width.5pt depth0pt}
\hskip-.2pt\vbox{\hrule height.5pt width7pt depth0pt}
\, }
\newcommand{\restr}{\hel}

\newcommand{\st}{\stackrel}
\newcommand{\mc}{\mathcal}
\newcommand{\ov}{\overline}

\newcommand{\sm}{\backslash}
\newcommand{\pt}{\partial}
\newcommand{\vhi}{\varphi}
\newcommand{\eps}{\varepsilon}
\newcommand{\dl}{\delta}
\newcommand{\e}{_\varepsilon}
\newcommand{\te}{\theta}
\newcommand{\sub}{\subset}
\newcommand{\oo}{\infty}
\newcommand{\loc}{_{\textnormal{loc}}}
\newcommand{\void}{\varnothing}
\newcommand{\nb}{\nabla}
\newcommand{\h}{\mathcal{H}}

\newcommand{\dw}{\downarrow}
\newcommand{\up}{\uparrow}
\newcommand{\cd}{\cdot}
\newcommand{\ncd}{\nabla\cd}
\newcommand{\we}{\wedge}
\newcommand{\longto}{\longrightarrow}
\newcommand{\lt}{\left}
\newcommand{\rt}{\right}
\newcommand{\bs}{\boldsymbol}

\newcommand{\un}{\bs{1}}
\newcommand{\Om}{\Omega}
\newcommand{\om}{\omega}
\newcommand{\ds}{\displaystyle}

\newcommand{\be}{\begin{equation}}
\newcommand{\ee}{\end{equation}}

\newcommand{\ub}{\bs{u}}
\newcommand{\vb}{\bs{v}}
\newcommand{\uv}{\underline{\bs{v}}}
\newcommand{\zb}{\bs{z}}
\newcommand{\wb}{\bs{w}}
\newcommand{\mb}{\bs{m}}
\newcommand{\xib}{\bs{\xi}}
\newcommand{\Phib}{\Phi}
\newcommand{\Psib}{\Psi}

\newcommand{\mus}{\nono\mu}

\newcommand{\alfs}{{\nono\alpha}}
\newcommand{\Psis}{{\nono\Psi}}
\newcommand{\As}{\nono{\mathcal{A}}}

\newcommand{\Es}{\nono{E}}
 
\def\XXint#1#2#3{{\setbox0=\hbox{$#1{#2#3}{\int}$} 
      \vcenter{\hbox{$#2#3$}}\kern-.5\wd0}}

\newcommand{\ghost}[1]{}

\DeclareMathOperator{\Id}{Id}

\DeclareMathOperator{\Lip}{Lip}

\DeclareMathOperator{\GL}{GL}
\DeclareMathOperator{\supp}{supp}

\DeclareMathOperator{\dist}{dist}

\DeclareMathOperator{\Span}{span}
\DeclareMathOperator{\curl}{curl}
\DeclareMathOperator{\rot}{{\bf \nono{ curl}}}
\DeclareMathOperator{\rotp}{\rot}
\DeclareMathOperator{\Var}{Var}

\DeclareMathOperator{\ent}{ENT}
\DeclareMathOperator{\entev}{ENT_{ev}}

\title[Unoriented Aviles--Giga functionals]{Compactness and structure of zero-states for unoriented Aviles--Giga functionals}
\date{\today}
\author{M. Goldman}
\address{Universit\'e de Paris and Sorbonne Universit\'e, CNRS, LJLL, F-75005 Paris, France}
\email{michael.goldman@u-paris.fr}
\author{B. Merlet}
\author{M. Pegon}
\address{Univ. Lille, CNRS, UMR 8524, Inria - Laboratoire Paul Painlev\'e, F-59000 Lille}
\email{benoit.merlet@univ-lille.fr}
\email{marc.pegon@univ-lille.fr}
\author{S. Serfaty}
\address{Courant Institute of Mathematical Sciences, 251 Mercer st, New York NY 10012,~USA}
\email{serfaty@cims.nyu.edu}

\begin{document}

\begin{abstract}
Motivated by some  models of pattern formation involving an unoriented director field in the plane, we study a  family of unoriented counterparts to the Aviles--Giga functional.

We introduce a   nonlinear $\curl$ operator for such unoriented vector fields as well as a family of even entropies which we call ``trigonometric entropies". Using these tools we show two main theorems which parallel some results in the literature on the classical Aviles--Giga energy. The first is a compactness result for sequences of configurations with uniformly bounded energies.  The second is a complete characterization of zero-states, that is, the limit configurations when the energies go to 0. These are Lipschitz continuous away from a locally finite set of points, near which they form either a vortex pattern or a  disclination with degree 1/2. The proof is based on a combination of regularity theory together with  techniques coming from the study of the  Ginzburg--Landau energy.\\
Our methods provide alternative proofs in the classical Aviles--Giga context.

 \bigskip

 \bigskip

\end{abstract}

\maketitle

 \section{Introduction}
 
In this paper we study unoriented variants of the two-dimensional Aviles--Giga functional. 
We first recall  the main features of  the classical (oriented) Aviles--Giga functional, which is nothing else than the Ginzburg--Landau energy restricted to curl-free vector fields, \textit{i.e.} gradients if the domain is simply connected.   
More precisely,  the functional
 is defined on the space of  vector fields $\ub:\Om\sub\R^2\to\R^2$ by 
\[
E\e(\ub):=\begin{cases}\dfrac\eps2 \ds\int_\Om |\nb \ub|^2\,  + \dfrac1{2\eps}\int_\Om (1-|\ub|^2)^2&\text{ if }\curl \ub=0,\\
\qquad+\oo&\text{otherwise}.\end{cases}
\]
 This  model, first  introduced in~\cite{AG96}, as well as some of its variants, appear in the modelling of various phenomena in materials science such as blistering of thin films, liquid crystals configurations 
 and magnetization orientation in ferromagnetic materials.  They  have attracted considerable attention in the mathematical literature 
 over the last twenty years, see \textit{e.g}. ~\cite{AG87,AG96,JK00,AdLM99,DKMO01,JP01, JOP02,RS01,ARS02,ALR02,dLO03,CdL07,Poli07,Ignat12,IM11,IM12,Lorent14} and  
 are  still the subject of active research  as witnessed by more recent articles~\cite{LP18,LLP20,GL20,Marconi20, Marconi21,LP21}.

The main question is to understand the behavior as $\eps\dw0$ of configurations with bounded energy (such as minimizers) and in particular to derive a $\Gamma-$limit of the energy.
The conjecture is that in the limit, the energy concentrates on line singularities corresponding to interfaces (``domain walls") in micromagnetics.
In full generality and despite substantial progress, this question is still open  to this date.\\
The first step of the program, which was carried in \cite{AdLM99,DKMO01}, is to prove strong $L^1$ convergence of sequences of bounded energy.
This shows in particular    that in the limit we obtain curl-free unit-norm vector fields. 
 The proof combines a compensated-compactness argument together with  the fact that  the energy controls  a certain  {\it entropy} production. The latter was already observed by~\cite{JK00} which introduced the first entropies for this problem. 
 This is inspired by the analysis of scalar conservation laws, observing that  
the eikonal equation
\be
\label{Intro_eikonal}
\curl \ub=0,\qquad |\ub|=1
\ee
 can be considered as a one-dimensional scalar conservation law $\partial_1 u_1 + \partial_2 (\sqrt{1-u_1^2})=0$. To be more specific,  an entropy is any mapping
 $\Phib\in C^\oo(\So,\R^2)$ such that 
 \begin{equation}\label{def:entropintro}\mu_\Phib[\ub]:=\ncd[\Phib(\ub)]=0
  \end{equation}
 for any {\it smooth} $\ub$ satisfying~\eqref{Intro_eikonal}. 
 For solutions  $\ub$ of \eqref{Intro_eikonal} obtained as limits of bounded energy configurations $\ub\e$ and  any entropy $\Phib$, the {\it entropy production} $\mu_\Phib[\ub]$ is
 typically not zero but 
 a {\it signed measure}, of mass controlled by $E\e(\ub\e)$ in the sense that 
 \begin{equation}\label{eq:control_entro_oriented}
  |\mu_\Phib[\ub]|(\Omega)\le C \liminf_{\eps\dw 0} E\e(\ub_\eps)
 \end{equation}
 for some positive constant $C$ depending on $\Phib$. In the particular case of the so-called Jin--Kohn entropies, \cite{JK00, AdLM99} proved the sharp inequality \eqref{eq:control_entro_oriented} with $C=1$,
 leading to a characterization of the  $\Gamma-$liminf (see also \cite{RS01,ARS02,IM11,IM12}). For $BV$ vector fields, this was complemented by a corresponding $\Gamma-$limsup construction in \cite{CdL07,Poli07}.
 However, as shown in \cite{AdLM99},  limit configurations are in general  not $BV$. In order to complete the program, it was therefore necessary to investigate further the fine 
 structure of configurations $\ub$ satisfying~\eqref{Intro_eikonal} such that $\mu_\Phib[\ub]$ is a signed measure for any entropy $\Phib$. A first step considered in ~\cite{JOP02}, was to study the case of
configurations such that $E\e(\ub\e)\to 0$ as $\eps\dw 0$, or more precisely, in regions of the domain where the energy does not concentrate and where $\mu_\Phib[\ub]$
vanishes by \eqref{eq:control_entro_oriented}. In this case,  the limiting configurations, called zero-states, must be Lipschitz continuous away from a locally finite singular set $S$. 
The singularities near the points of $S$ must be of  vortex-type (\textit{i.e.} $u(x) =\pm\, x/|x|$ locally, up to an origin-shift). Recently, \cite{LP18,LLP20} proved that the same conclusion can be obtained under
the weaker assumption that the entropy production coming from the  Jin--Kohn entropies vanishes.\\
An important further step was obtained in~\cite{dLO03} (see also \cite{ALR02}), where  it is shown that configurations of finite energy   share some of the characteristic properties of $BV$-mappings.
In particular, it is possible to define  a countably-rectifiable jump set $J_{\ub}$ with  weak traces $\ub^\pm$ on both sides.
Probably the main open question to complete the proof of the $\Gamma-$convergence is to show that the entropy production is concentrated on $J_{\ub}$, see \cite{GL20,Marconi20,Marconi21,LP21} for recent progress on this question. 
\medskip
 
 Motivated by models for the  formation of stripe patterns~\cite{EINP03,EV09,ZANV21}, two of the
 authors started to study in~\cite{MS21} some {\it unoriented variants} of the Aviles--Giga functional,
 \textit{i.e.} variants in which $\ub$ and $-\ub$ are identified.
More precisely, working in the  $SBV$ setting  (\textit{i.e.} $BV$ functions whose differential has no Cantor part, see \cite{AFP}) they consider the energy 
 \be\label{AGeps}
E'\e(\ub\otimes \ub):=\begin{cases}
\dfrac\eps2 \ds\int_\Om |\nb_a \ub|^2\,  + \dfrac1{2\eps}\int_\Om (1-|\ub|^2)^2 
    & \text{if }\begin{cases}&\ub\in SBV(\Om,\R^2),\\ &\ub^++\ub^-=0\text{ on }J_{\ub}\\ &\text{and }\curl_a \ub=0,\end{cases}\\
\qquad+\oo&\text{otherwise,}\end{cases}
\ee
where $\nb_a\ub=(\partial_1 \ub,\partial_2 \ub)$ denotes the absolutely continuous part of the distributional gradient $\nb \ub$, $\curl_a \ub=(\partial_1 u_2)_a-(\partial_2 u_1)_a$, $J_{\ub}$ is the jump set of $\ub$ and $\ub^\pm$ the traces of $\ub$ on $J_{\ub}$. 
This preserves  the original model as much as possible while allowing non-orientable fields. One can check that this functional is unambiguously defined as a function of $\ub \otimes \ub$, since if $\tilde \ub\otimes \tilde \ub=\ub\otimes\ub$ then $\tilde \ub(x)=\pm\ub(x)$ almost 
everywhere in $\Om$ and  $|\nb_a \tilde \ub|=|\nb_a \ub|$.\\

As shown in  \cite{MS21}, the passage from the oriented to the unoriented setting has a dramatic impact on the properties of configurations with moderate energy. 
First, the optimal jump profiles are not always one-dimensional and are thus difficult to precisely characterize 
(in particular they are different from  the two-dimensional ``cross-tie" patterns found in its micromagnetics variants~\cite{ARS02}).
Second and maybe more importantly,  the curl-free constraint  may be  lost in the limit $\varepsilon \dw 0$. This shows in particular that following the
program described above will be very challenging in the unoriented case.\\


 Nevertheless, the aim of this paper is to perform the first two steps and  prove compactness, which happens, maybe surprisingly, despite the possible loss of the curl-free condition, 
 as well as to investigate the 
  structure of zero-states,  thus providing  a parallel to the results of~\cite{DKMO01} and~\cite{JOP02}.  \\
  
 Since we have no control on $\ub$ but only on $\uv=\ub\otimes\ub$, an important preliminary step is to express the energy $E'_{\eps}$ in terms of $\uv$. 
 To this aim, we introduce of good notion of curl for unoriented configurations (see Definition \ref{def_rot} and \eqref{eq:rotp} below) and  denote  it by $\rotp \uv$. It turns out that a convenient requirement is that   
  $\rotp \uv= (\curl \ub) \ub$ for smooth $\ub$ with $\uv=\ub\otimes \ub$.   It is thus a vector-valued and {\it nonlinear} operator. We also need to define the entropy production in terms of $\uv$. For $\Phib$ an even entropy, \textit{i.e.} $\Phib(-z)=\Phib(z)$ for $z\in \So$ and $\uv=\ub\otimes \ub$, we define 
  \[
   \mus_{\Phib}[\uv]:=\ncd[\Phib(\ub)].
  \]
Our first main result is then the following.
\begin{theorem}\label{theo:introcomp}
 Let $\Omega\subset \R^2$ be a domain of finite area and  $\eps_k\dw 0$. If $\uv_k=\ub_{k}\otimes \ub_{k}$ is such that $\sup_k E'_{\eps_k}(\uv_{k})<\infty$, then there exists
 $\uv=\ub\otimes \ub$ with $\ub\in L^1(\Omega,\So)$ 
 such that up to extraction, $\uv_k\to \uv$ in $L^1$. Moreover, for every even entropy $\Phib$, the corresponding entropy production $\mus_{\Phib}[\uv]$ is a signed measure with 
 \[
  |\mus_{\Phib}[\uv]|(\Omega)\le C \liminf_{k\to \infty} E'_{\eps_k}(\uv_{k})
 \]
for some constant $C>0$ depending on $\Phib$.
\end{theorem}
Let us point out that we actually prove a more general compactness result which allows for a relaxation of the curl-free condition (see Theorem \ref{thm:compmain}).\\
 
Theorem \ref{theo:introcomp} shows in particular that configurations of vanishing energy are {\it zero-states} in the sense that $\uv=\ub\otimes \ub$ for some $\ub$ of unit length with
$\mus_{\Phib}[\uv]=0$ for every even entropy. Our second main result  proves  that in the unoriented setting the structure of zero-states is very similar to that described in~\cite{JOP02} in the classical oriented setting.
The only difference is that point singularities may be vortices but also $1/2$-disclinations (we can also interpret a vortex  as two glued $1/2$-disclinations so that, essentially, these latter are the only type of singularities).
\begin{theorem}[Structure of zero-states]~\\
\label{thm_0nrjw}
Let $\uv=\ub\otimes \ub$ be a zero-state. Then $\rotp \uv=0$ and  there exists a locally finite set $S\sub \Om$ such that:
\begin{enumerate}[(i)] 
\item $\uv$ is locally Lipschitz continuous in $\Om\sm S$,
\item  for $x\in\Om\sm S$, $\uv=\uv(x)$ on the connected component of $[x+\R\ub(x)]\cap [\Om\sm S]$ which contains $x$,
\item for every $B:=B_{r}(x^0)$ such that $2B:=B_{2r}(x^0)\sub \Om$ and $2B\cap S=\{x^0\}$, we can choose $\ub$ such that 
\begin{enumerate}[(a)]
\item either $\ub(x)=(x-x^0)/|x-x^0|$ in $B\sm\{x^0\}$ {(see Figure~\ref{Fig_vortex})},
\item or there exists $ \xib\in \So$ such that {(see Figure~\ref{Fig_disclination})},\smallskip
\begin{itemize}
\item[$\circ$] $\ub(x)=(x-x^0)/|x-x^0|$ in $\ds\lt\{x\in B\sm\{x^0\} : (x-x^0)\cd\xib\ge0\rt\}$,\smallskip
\item[$\circ$] $\uv$ is  Lipschitz continuous in $\ds\lt\{x\in B\sm\{x^0\}: (x-x^0)\cd\xib\le0\rt\}$.
\end{itemize}
\end{enumerate}
\end{enumerate}
\end{theorem}

\begin{figure}[ht]
    \begin{minipage}{0.48\textwidth}
        \begin{center}
	\captionsetup{width=.95\textwidth}
        \includegraphics[width=0.9\textwidth]{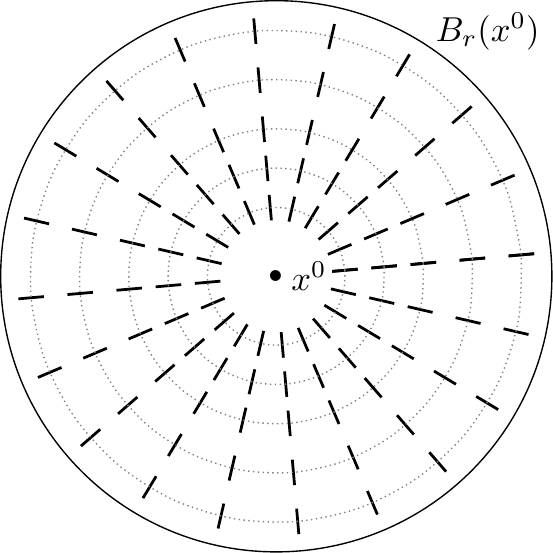}
        \caption{Case (iii)(a) a vortex. \\~}
	\label{Fig_vortex} 
	\end{center}
    \end{minipage}\hfill
    \begin{minipage}{0.48\textwidth}
          \begin{center}
	\captionsetup{width=.95\textwidth}
        \includegraphics[width=0.9\textwidth]{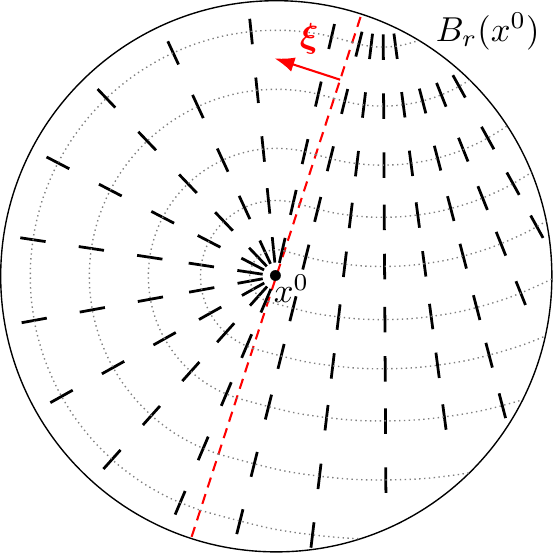}
        \caption{Case (iii)(b) a 1/2-discli-nation.}
	\label{Fig_disclination}
	\end{center}
    \end{minipage}
	\vspace{2pt}
Structure of a zero-state near a singularity $x^0$, as in Theorem~\ref{thm_0nrjw}. Line segments represent the director field $\pm\ub$. Dotted curves are orthogonal to this field.
\end{figure}

Let us stress  that an important point of Theorem \ref{thm_0nrjw} is that zero-states are curl-free. As already alluded to, this is in general not true for arbitrary limit configurations (see  \cite{MS21}). 
In particular, this shows that the creation of curl must come with a cost. \\

 To sum up, our main accomplishments are the following:

\begin{enumerate}[(a)] 
\item 
 We introduce a good notion of curl for unoriented configurations, which we denote $\rot$, and  which extends the usual notion of $\curl$ in the sense that $\rot (\ub\otimes\ub)=(\curl\ub)\ub$. This $\rot$ operator is vector-valued and {\it nonlinear}.

\item We introduce a new family of entropies, which we call the {\it trigonometric entropies} and which have remarkable arithmetic properties.  These are at the core of the proofs of both theorems. 
To the best of our knowledge, albeit being quite natural (in particular they contain the Jin--Kohn entropies, see Remark \ref{rem:JinKohn}),  these entropies have never been used so far. We thus anticipate that they could also prove useful for other related problems.    In order to use entropies in the unoriented setting, these  have to be {\it even} in the usual sense. This is not the case for the standard entropies of~\cite{JK00,DKMO01}. We are able to show that, despite this constraint which reduces the admissible family, controlling just the entropy production for the family of  (even) trigonometric entropies is enough to recover compactness.

\item We prove a parallel structure theorem to~\cite{JOP02}, precisely that zero-states (in particular limits of vanishing energy configurations obtained by the compactness result)  are locally Lipschitz continuous away from a locally finite singular set. Moreover, every singular point corresponds either to a vortex $x/|x|$, or to a disclination with degree $1/2$, see Figures~\ref{Fig_vortex},~\ref{Fig_disclination}. 
 Because of the unoriented situation, the kinetic formulation approach employed in~\cite{JOP02} is no longer available. Instead, the route we follow is to  prove $W^{1,p}$ regularity of zero-states. This is inspired by ideas from~\cite{LP18}. Using results from Ginzburg--Landau theory (see \cite{BBH,SS07,AlicPon}) this allows to identify the location of the vortices (or disclinations).  At this point, thanks to the $W^{1,p}$  regularity we can 
 follow the characteristics  (here its level lines) in a classical sense to conclude on the geometric structure.
\end{enumerate}

We now give an outline of the proofs of Theorem \ref{theo:introcomp} and Theorem \ref{thm_0nrjw} together with more precise definitions of the objects under consideration. 

 \subsection{Main definitions and outline of the proofs}
 \subsubsection*{Complex representation}
 Inspired by~\cite{GMM20} where a related unoriented functional of Ginzburg--Landau type (thus without constraints on the curl) was considered, 
 we find it actually convenient  not to work  with tensor products $\ub \otimes \ub$ but rather in complex number representation. We will use a bold font for elements of $\R^2$ 
 and a regular font for the corresponding elements of $\C$, in the sense that 
 $\ub=(u_1,u_2)$ is identified with $u=u_1+iu_2\in \C$. 
 
 We now observe that  for $u, \tilde u\in \C$ we have 
 \[
\tilde u^2 = u^2 \quad\Longleftrightarrow\quad \tilde u=\pm u \quad\Longleftrightarrow\quad\tilde{\ub}\otimes\tilde{\ub}=\ub\otimes\ub.   
\]
We can thus, indeed identify unoriented vector fields $\uv=\ub\otimes \ub$ with  complex-valued functions $v=u^2$.   In the sequel, we heavily use the multiplicative structure of $\C$.  From now on we also identify all the quantities depending on $v$ and $\uv$ (we write $\rot v$ for $\rot \uv$ and so on). Notice that  if $\uv=\begin{pmatrix} \uv_{11} & \uv_{12}\\ \uv_{12} & \uv_{22}\end{pmatrix}$ and $v=v_1+i v_2$ we have the correspondence
 \begin{equation}\label{v=w}
v_1=\uv_{11}-\uv_{22} \qquad v_2= 2\uv_{12} \qquad \textrm{and} \qquad |v|={\rm Tr\,} \uv.  
 \end{equation}
 In order to define unambiguously a square-root 
$\sigma$ on $\C$ we set  
\begin{equation}\label{defsigma}
\sigma(r e^{i\te})=\sqrt{r}e^{i\te/2}\quad\text{ for }-\pi< \te\le\pi\text{ and }r\ge0.
\end{equation}
So, $\sigma$ is a right inverse on $\C$ for $\Pi:z\in\C\mapsto z^2$ and a left inverse for $\Pi$ on $\{re^{i\te}:r\ge0,\,-\pi/2<\te\le\pi/2\}$. In particular, we have $\sigma(v)^2=v$.\\

 \begin{remark}
 In the study of nematic liquid crystals and in other fields, directors $\{\pm \ub\}$ with $\ub\in\So$ are usually represented as $Q_{\ub}=\ub\otimes\ub -(1/2) \Id$. 
 Taking local averages, the theory is extended to $Q-$tensors: mappings taking values into the space of $2\times2$ symmetric traceless matrices. Here, in the energy~\eqref{AGeps},
 the vector fields $\ub$ are not supposed to have unit length and a representation by $Q$-tensors is not straightforward (the trace of $\ub\otimes\ub$ is not prescribed).  
Moreover, it is not clear how to properly define a curl operator for general $Q-$tensors (we could extend formula \eqref{eq:rotp} below but many other choices are possible). 
 \end{remark}
 
 \subsubsection*{Unoriented curl and energies}
 We may now define our  unoriented curl operator, which has the particularity of being  nonlinear and vector-valued. It is defined so that if $v=u^2$ with $u$ smooth, then $\rot v= (\curl \ub) \ub$.

\begin{definition}\label{def_rot}
For $v\in W^{1,1}\loc(\Om,\C)$, $\rot v$ denotes the measurable vector field $\Om\to\R^2$ which vanishes on the set $\{x\in\Om:v(x)=0\}$ and is defined by $\rot v:=(\rho_1,\rho_2)$ elsewhere with
\begin{align*}
\rho_1&:=\dfrac14\lt( \dfrac{v_1}{|v|} + 1\rt)(\pt_1v_2 - \pt_2 v_1)  -\dfrac14\dfrac{v_2}{|v|} (\pt_1v_1 + \pt_2 v_2),\\
\rho_2&:=\dfrac14\dfrac{v_2}{|v|} (\pt_1v_2 - \pt_2 v_1)  +\dfrac14 \lt( \dfrac{v_1}{|v|} -1\rt) (\pt_1v_1 + \pt_2 v_2).
\end{align*}
This can be rewritten in compact form as
\begin{equation}\label{def:rotcomp}
\rot v =\dfrac14 \lt(\begin{array}{cc}
\dfrac{v_1}{|v|} + 1\ & -\dfrac{v_2}{|v|}\ \\ \\ \!\dfrac{v_2}{|v|} & \dfrac{v_1}{|v|} - 1
\end{array}\rt)
\binom{\nb^\perp\cd\vb}{\ncd\vb},
\end{equation}
where  $\nb^\perp:=(-\pt_2,\pt_1)$. Notice that $\rot v\in L^1\loc(\Om,\R^2)$ defines a distribution.
\end{definition}
\begin{remark}
 Using for instance \eqref{v=w} it is not hard to see that 
 \begin{equation}\label{eq:rotp}
  \rotp \uv=\frac{1}{2 {\rm Tr\,}(\uv)} \uv\lt(2 (\nabla\cdot \uv)^{\perp} +\nabla^\perp( {\rm Tr\,}(\uv))\rt)
\end{equation}
where $\ub^\perp:=(-u_2,u_1)$ and 
\[
 \nabla \cdot \uv=(\nabla \cdot \uv_1,\nabla \cdot \uv_2)
\]
is the column-wise divergence of $\uv$.
\end{remark}

\begin{remark}
We cannot use only one component of $\rot(u^2)=(\rho_1,\rho_2)$ to retrieve $\curl u$ or even the condition $\curl u=0$. Indeed, if $u$ takes only real values then $v=u^2$ satisfies $v_1-|v|=v_2= 0$, hence $\rho_2=0$ whereas $u$ is curl-free only if $\pt_2u= 0$. Symmetrically, if $u$ takes values in $i\R$ then  $v_1+|v|=v_2= 0$ and $\rho_1=0$ while $\curl u$ vanishes  only if $\pt_1u= 0$.
\end{remark}
We can now use the $\rot $ operator to encode the curl-free constraint  and replace the functional $E'\e$ defined in~\eqref{AGeps} by 
\begin{equation}\label{AGepsprime}
E\e''(v):=
\begin{cases}
\dfrac\eps2 \ds\int_\Om \dfrac{|\nb v|^2}{4|v|}  + \dfrac1{2\eps}\int_\Om (1-|v|)^2
  &\text{if }v\in W^{1,1}(\Om,\C)\text{ and }\rot v=0,\smallskip\\
\qquad+\oo
  &\text{otherwise.}
\end{cases}
\end{equation}
%
By convention $|\nb v|^2/(4|v|)=0$ wherever $v=0$ (recall that $\nb v=0$ a.e. on $\{v=0\}$). We will show in Proposition \ref{prop_equivalence_modeles}  that the functionals $E'\e$ and $E''\e$ are indeed equivalent. \\

Drawing analogy with both the Aviles--Giga functional and micromagnetics models \cite{DKMOS01,DKMO02,DKMO05,RS01,ARS02}, we will actually consider  a larger class of  energies where  the curl-free constraint is relaxed. We will consider two possible relaxations. The first one is meant to 
allow dislocations at small scales while penalizing their energy, the other contains an analogue of the stray-field energy in micromagnetics.\smallskip\\
We assume here that $g $ and $W$ are measurable functions from $\C$ to $[0, +\infty]$, that $\lambda^0\e,\lambda^1\e\in [0,+\infty]$ and that there exists a constant $\kappa>0$ such that  
\begin{equation}
\label{conditionsgw}
\lt\{ \begin{array}{rll}g&\ge \kappa &\text{in some neighborhood of }\So,\smallskip\\
W(v) &\ge \kappa \min\lt( (1-|v|)^2,|1-|v|| \rt)\quad&\text{for }v\in\C,\smallskip\\
\max(\lambda^0\e,\lambda^1\e)&\ge \kappa&\text{for }\eps>0.
\end{array}\rt.\end{equation}
We then define  
\begin{align}
\label{Eseps0}
\Es\e(v)&:=\eps \ds\int_\Om g\lt(v\rt) |\nb v|^2\,  + \dfrac1{\eps}\int_\Om  W(v)\,+ \lambda\e^0\int_\Om |\rot v|+\dfrac{\lambda\e^1}{\eps}\lt\|\rot v\rt\|_{H^{-1}(\Om)}^2.
\end{align}
Notice that $E\e''$ corresponds to the choice $g(v)=1/(8|v|)$, $W(v)=(1-|v|)^2/2$ and $\max(\lambda^0\e,\lambda^1\e)=+\infty$.

\subsubsection*{Unoriented and trigonometric entropies}
We define entropies as mappings $\Phib\in \C^{\infty}(\So, \C^2)$ such that the real and imaginary part of $\Phib$ are usual entropies \textit{i.e.} they satisfy \eqref{def:entropintro}. We fix an arbitrary function $\chi_0\in C^{\infty}_c([0,+\infty),[0,1])$ with $\chi_0(1)=1$ and $\supp \chi_0\subset (1/2,2)$. With a slight abuse of notation we identify any entropy $\Phib$ with its extension to $\C$ by
\[
 \Phib(z)=\chi_0(|z|)\Phib\lt(\frac{z}{|z|}\rt).
\]
For any $v\in L^1(\Omega,\C)$, and any even entropy $\Phi$, we define the {\it entropy production}
\begin{equation}\label{def:entroprod}
 \mus_{\Phib}[v]:=\nabla\cdot\lt[\Phi(\sigma(v))\rt].
\end{equation}

In order to motivate the definition of the trigonometric entropies, let us recall from \cite{DKMO01} that setting $\T:=\R/(2\pi \Z)$, we can associate to any $\lambda\in C^\oo(\T,\C)$  the entropy $\Ups[\lambda]:=\Phib\in C^\oo(\So,\C^2)$ defined by 
\begin{equation}\label{def:Gamma}
\Phib(e^{i\te}):= \lambda(\te) \binom{-\sin\te}{\cos\te} -\lambda'(\te) \binom{\,\cos\te\,}{\,\sin\te\,}.
\end{equation}
 Moreover, if $\lambda$ is $\pi-$antiperiodic then $\Phib$ is even. 

 \begin{definition}\label{def_Phi_n}
For $n\in\Z$, we set  $\Phib^n:=\Ups[2i e_n]$, where $e_n$ is the trigonometric monomial $\te\in\R\mapsto e^{in\te}\in\C$. We call these functions  \emph{trigonometric }entropies. 
More explicitly, for $n\in \Z$ and $z\in \C$,
\[\Phib^n(z)=(n-1)z^{n+1} \binom1{-i} + (n+1) z^{n-1} \binom1i. \]
 \end{definition}
Notice that  if $n$ is odd then $\Phib^n$ is even. The main properties of these entropies that we will use are the very favorable algebraic expressions of 
\[
 \Phib^n(z)\wedge \Phib^{-n}(z') \qquad \textrm{and} \qquad \lt[\Phib^n(z)-\Phib^n(z')\rt]\wedge \lt[\Phib^{-n}(z)-\Phib^{-n}(z')\rt],
\]
from Lemma \ref{lem_Phinwedges}.

\begin{remark}\label{rem:JinKohn}
 Defining the Jin--Kohn entropies  on $\So$ by
$\Sigma_1(z):=(z_2(1- (2/3)z_2^2), z_1(1-(2/3)z_1^2))$ and $\Sigma_2(z):=(2/3)(z_1^3,-z_2^3)$ (see
\cite{JK00}) we can easily check that $\Phib^{\pm 2} = 6\Sigma_2 \pm 6 i\, \Sigma_1$.
\end{remark}

\subsubsection*{Sketch of proof of Theorem \ref{theo:introcomp}}
The proof of Theorem \ref{theo:introcomp} (for $\Es\e$) follows closely the  strategy of the proof of its oriented counterpart in \cite{DKMO01}.
The first step (see Proposition \ref{prop_muPhiH-1}) is to prove that the energy controls the entropy production in the sense that for every even entropy $\Phib$, 
there exists a constant $C\ge 0$ depending on $\Phib$ such that for every $\zeta\in C^1_c(\Omega,\C)$,
\begin{equation}\label{eq:entropyprodintro}
\lt|\int_\Omega \mus_{\Phib}[v] \zeta\rt|\le C\lt(\Es\e(v)\|\zeta\|_\infty + \eps^{1/2}\Es\e^{1/2}(v)\|\nabla \zeta\|_2\rt).
\end{equation}
Thanks to  \cite{Murat78,Murat81,Tartar79} this yields that for sequences $(v_k)$ of bounded energy, $( \mus_{\Phib}[v_k])$ is compact in $H^{-1}\loc(\Omega)$. Using the div-curl lemma of Murat and Tartar, see \cite{Murat78,Tartar79}, we conclude that for every pair of even entropies $\Phi^a$, $\Phi^b$ we have as weak limits in $L^2$,
\begin{equation}\label{divcurlintro}
 \lim_{k\up \infty} \lt[\Phib^a(\sigma(v_k))\wedge \Phib^b(\sigma(v_k))\rt]=\lt[ \lim_{k\up \infty} \Phib^a(\sigma(v_k))\rt]\wedge\lt[\lim_{k\up \infty} \Phib^b(\sigma(v_k))\rt].
\end{equation}
If we now consider the Young measure  $\nu_x\otimes\mathcal{L}^2$ generated by $(v_k)$ ($\mathcal{L}^2$ denotes the Lebesgue measure) this translates into 
\[
 \int_{\So} \Phib^a(\sigma(z))\wedge \Phib^b(\sigma(z)) d\nu_x(z)=\lt[\int_{\So} \Phib^a(\sigma(z))d\nu_x(z)\rt]\wedge \lt[\int_{\So}\Phib^b(\sigma(z)) d\nu_x(z)\rt].
\]
The question is then to understand if the class of entropies we have at our disposal is rich enough to conclude that $\nu_x$ must be a Dirac mass (which then classically implies strong convergence of $(v_k)$, see \cite{MullerLecture} for example). It is at this point that our proof departs from the one of \cite{DKMO01}. Indeed, to reach the conclusion in the oriented case~\cite{DKMO01} used (after an approximation argument) the following entropies for $\xi\in\So$,
\begin{equation}\label{Phixi}
\Phib^\xi(z):=\begin{cases}\xib&\text{when }\xib\cd\zb^\perp>0,\\0&\text{when }\xib\cd\zb^\perp\le0.\end{cases}
\end{equation}
Alternatively, in \cite{AdLM99} the Jin--Kohn entropies of  \cite{JK00} are used for a similar conclusion. None of these entropies are even and therefore available in the unoriented setting.
We use instead the trigonometric entropies  $\Phib^n$ for odd $n$ to prove that the first Fourier coefficient of $\nu_x$ is of modulus equal to $1$. Since $\nu_x$ is a
probability density on $\So$, this implies that it is a Dirac mass (see Lemma \ref{lem_dirac}).

\begin{remark}~
\begin{enumerate}[(1)]
\item
 Notice that we can also obtain an alternative proof to the compactness results of \cite{DKMO01,AdLM99} in the oriented case by using instead the trigonometric entropies $\Phib^n$ with $n$ even.
\item 
On the contrary, if we replace $v=u^2$ by $v=u^q$ for some $q\ge 3$ in the model \textit{i.e.} we identify the elements of $\{e^{2i k\pi/q} u : 0\le k\le q-1\}$,  then the space of entropies with the respective symmetry reduces to the space of constant mappings $\So\to\C^2$.
In this case it turns out that the compactness result analogous to Theorem~\ref{thm:compmain} does not hold  and the corresponding $\Gamma-$limit is trivial (see \cite{MS21}). 
%
\end{enumerate}
\end{remark}

\subsubsection*{Sketch of proof of Theorem \ref{thm_0nrjw}}
Our proof of Theorem \ref{thm_0nrjw} is totally different from the proof in \cite{JOP02} for the oriented case.
Indeed, that proof is almost entirely based on the use of the entropies $\Phib^\xi$ defined in \eqref{Phixi} in order to follow in a weak sense the characteristics of the eikonal equation (which are lines).
In the unoriented case, we   build instead on a method of Lorent and Peng~\cite[Theorem~4]{LP18} (see also~\cite[Lemma~7]{LLP20}) which in turn is inspired by the earlier work \cite{Sverak} of \v{S}ver\'{a}k on differential inclusions.
We apply it  in Proposition \ref{lem_reg} with the trigonometric entropies to show that
if $v$ is a zero-state then $v\in W^{1, 3/2}_{\rm loc}(\Omega)$ (actually $W^{1,p}_{\rm loc}$ for every $p<2$, 
see Remark \ref{rem:regzero}) and for every open set $\omega\subset \sub \Omega$ and $|h|\ll1$,
\begin{equation}\label{eq:reg_estim_intro}
 \int_{\omega} |v(x+h)-v(x)|^2 dx\le C |h|^2 \ln(1/ |h|).
\end{equation}
From $v\in W^{1, 3/2}_{\rm loc}(\Omega)$ and the vanishing of any non-trivial even trigonometric entropy we recover the condition $\rot v=0$. Before commenting on the implications of \eqref{eq:reg_estim_intro}, let us give a sketch of the proof. For the sake of simplicity we focus on the first non-trivial even trigonometric entropy $\Phib^3$ and derive the weaker Besov estimate:
\begin{equation}\label{Besovintro}
  \int_{\Omega} |D_h v|^4 \zeta^2\le C |h|^{\frac{4}{3}},
\end{equation}
that is $v\in B^{1/3}_{4,\infty,{\rm loc}}(\Omega)$. This is the exact analog of \cite[Theorem~4]{LP18}
replacing the Jin--Kohn entropies by the trigonometric entropies $\Phib^{\pm 3}$.\\
Letting $D_h f(x):=f(x+h)-f(x)$,
the main point is that as a consequence of Lemma~\ref{lem_Phinwedges},
\begin{equation}\label{eq:Dhv4}
 |D_h v|^4\le \frac{C}{2i} \lt(D_h[\Phib^3(\sigma(v))]\wedge D_h[\Phib^{-3}(\sigma(v))]\rt).
\end{equation}
We then use once again the div-curl structure of the right-hand side (recall \eqref{divcurlintro}). Indeed,  we notice  that $\Phib^3(\sigma(v))$ is divergence free (since $\mus_{\Phib^3}[v]=0$) to find a Lipschitz
function $F$ with $\Phib^3(\sigma(v))=\nabla^\perp F$.
For every smooth test function $\zeta$ we then have
\begin{align*}
 \int_{\Omega} |D_h v|^4 \zeta^2&\le C\lt| \int_\Omega \lt[\nabla^\perp D_h F\rt]\wedge \lt[D_h[\Phib^{-3}(\sigma(v))]\rt] \zeta^2\rt|\\
 &=C\lt| \int_\Omega \lt[\nabla D_h F\rt]\cdot \lt[D_h[\Phib^{-3}(\sigma(v))]\rt]\zeta^2\rt|\\
 &\le C \int_\Omega  |D_h F|  \lt|D_h[\Phib^{-3}(\sigma(v))]\rt|  |\zeta| |\nabla \zeta|,
\end{align*}
where we used integration by parts and $\nabla \cdot \lt[D_h[\Phib^{-3}(\sigma(v))]\rt]=0$. Since both $F$ and $\Phib^{-3}\circ \sigma$ are Lipschitz continuous, we find by H\"older inequality,
\[
 \int_{\Omega} |D_h v|^4 \zeta^2\le C |h| \lt(\int_{\Omega} |D_h v|^4 \zeta^2\rt)^{\frac{1}{4}}.
\]
After simplification this yields \eqref{Besovintro}.\\

In order to improve from the Besov regularity to the stronger $W^{1,p}$ estimate and to~\eqref{eq:reg_estim_intro}, we need to  implement a refined version of this argument involving {\it all} the even trigonometric entropies. 
\\

Defining the Ginzburg--Landau  energy  in the open set $\omega$ (see \cite{BBH, SS07}) as 
\[
 \GL_\eta(u;\om):=\dfrac{1}{2}\int_{\omega} |\nabla u|^2 +\frac{1}{4\eta^2}\int_{\omega}(1-|u|^2)^2
\]
we deduce in Lemma \ref{lem_GL} from \eqref{eq:reg_estim_intro} that for every mollification $v_\eta$ of $v$ we have
\begin{equation}\label{GLsmall}
 \GL_\eta(v_\eta;\om)\le C \ln(1/\eta).
\end{equation}
From the theory of Ginzburg--Landau vortices and in particular \cite{AlicPon} we find that $v$ has degree zero outside a locally finite set $S$.
In particular, we can use the theory of lifting for Sobolev maps \cite{BMP,Demengel} to find locally outside of $S$ a curl-free square-root of $v$ with Sobolev regularity.
At this point we can follow the characteristics in a classical sense to conclude on the geometrical structure.

\begin{remark}
 As for Theorem \ref{theo:introcomp}, our proof of Theorem \ref{thm_0nrjw} can be used to give a new proof in the classical (oriented) Aviles--Giga setting.
\end{remark}

\begin{remark}\label{rem:remregintro}
 Let us point out that in order to conclude that the number of singularities is locally finite, it is crucial to obtain a sharp estimate in \eqref{eq:reg_estim_intro}. 
 Indeed, even a logarithmically failing estimate would lead to a potentially infinite number of vortices. In particular, if $(v_k)$ is such that $v_k\to v$ and  $\Es_{\eps_k}(v_k)\to 0$, \eqref{GLsmall} should be compared with the much weaker bound 
 \[
  \GL_{\eps_k}(\om;v_k)=o(1/\eps_k).
 \]
 This work is the first where the connection between Ginzburg-Landau vortices and zero-states of the Aviles-Giga energy is made \textit{a priori} rather than \textit{a posteriori}.
\end{remark}

\subsection{Organization of the paper}
In Section~\ref{Seccurl} we show the equivalence between considering the variable $u$ and $v=u^2$ both for the curl-free constraint and for the energy.
In Section~\ref{Sentrop_prod}, we collect all the properties of entropies and entropy productions that we need.  
In Section~\ref{SPhin}, we study the family of trigonometric entropies and establish some of their properties.  
Compactness issues are dealt with in Section~\ref{Scomp}, in particular the general compactness result, Theorem~\ref{thm_comp}, is established. The structure of zero-states is obtained in Section~\ref{Sstructure_zero_states}. We indicate throughout the article how our methods apply to the classical Aviles--Giga functional.


\subsection{Conventions and notation}
We identify the target spaces $\R^2$  with $\C$ and recall the following conventions.
\begin{enumerate}[(i)]
\item 
$\So:=\{z\in \C:|z|=1\}$.

\item For $u=u_1+iu_2\in\C$  with $u_1,u_2\in\R$, we write $\ub=(u_1,u_2)\in\R^2$. This applies to functions. $u:\Om\to\C$ corresponds to $\ub:\Om\to\R^2$ with $\ub(x)=(u_1(x),u_2(x))$ and $u(x)=u_1(x)+iu_2(x)$.\\
However, we do not use bold fonts for the elements of the domain $\Om\sub\R^2$: we write $x=(x_1,x_2)\in\Om$.

\item We recall that $\sigma$ is defined by $\sigma(r e^{i\te})=\sqrt{r}e^{i\te/2}$  for $-\pi< \te\le\pi$ and $r\ge0$.
\end{enumerate}
Throughout the paper, $\Om$ is an open domain of $\R^2$ and  $\om\sub\sub\Om$ means that $\om$ is an open set with $\ov\om\sub\Om$ compact.\medskip


\noindent
For $\zb\in\C^2$, $\zb^\perp:=(-z_2,z_1)$. We also denote $\nb^\perp=(-\pt_2,\pt_1)$.\medskip

\noindent
For $\zb^{a}=\binom{z_1^a}{z_2^a},\zb^{b}=\binom{z_1^b}{z_2^b}\in\C^2$, we denote by $\zb^a\we\zb^b$ the wedge product defined by 
\[
\zb^{a}\we\zb^{b}=\det(\zb^{a},\zb^{b}) = (\zb^{a})^\perp\cd\zb^{b}= z^{a}_1z^{b}_2-z^{a}_2z^{b}_1.
\]

\noindent
For $\phi:\om\sub\R^2\to \C^2$, $\ncd \phi:=\pt_1\phi_1+\pt_2\phi_2$.\medskip

\noindent
For $u:\om\sub\R^2\to \C$, $\curl u:=\curl \ub=\nb\we\ub=\pt_1 u_2-\pt_2 u_1=-\ncd \ub^\perp=\nb^\perp\cd\ub$. \medskip


\noindent 
$|z|$ denotes the modulus of the complex number $z$, and $|x|$ the Euclidean norm of $x\in\R^2$. For complex vectors $\wb\in \C^2$, we write $\|\wb\|_{\ell^2(\C^2)}$ for the Euclidean norm $\sqrt{|w_1|^2+|w_2|^2}$. The associated bilinear dot products are denoted by ``\,$\cdot$\,''. \medskip


\noindent 
For $u\in BV\loc(\Om)$ (see~\cite{AFP}), we denote by $D_au$ the absolutely continuous part of its differential and by $D_cu$ its Cantor part. Similarly, $\nb_a u$ and $\curl_a u$ denote the absolutely continuous part of $\nb u$ and $\curl u$. The jump set of $u$ is denoted by $J_u$ and $\nu_u$ is a unit normal vector to $J_u$. The traces of $u$ at some point $x\in J_u$ are defined with the convention $u^\pm:=\lim_{t\dw0}u(x\pm t\nu_u(x))$.
\medskip

\noindent
We denote by $SBV^2(\Om)$, the subspace of elements $u\in SBV(\Om)$ such that $\nb_a u\in L^2(\Om)$. Notice that unlike~\cite{AFP} we do not impose the condition $\mathcal{H}^1(J_u)<+\infty$.



\section{Equivalence of the models}\label{Seccurl} 
We first show that $\rot (u^2)$ allows to recover $\curl_a u$: 
\begin{lemma}
\label{lem_curlrot}
Let  $u\in SBV^2\loc(\Om,\C)$, such that $(u^+)^2=(u^-)^2$ $\h^1$-almost everywhere on $J_u$, then $u^2\in W^{1,1}\loc(\Om,\C)$ and 
\[
\rot (u^2)= (\curl_a u)\ub\qquad\text{ in }L^1\loc(\Om).
\]

\noindent
In particular (with the convention $\ub/|u|^2=0$ wherever $\ub$ vanishes), there holds
\[
\curl_a u = \rot (u^2)\cd\dfrac{\!\!\ub\,}{|u|^2}\qquad\text{ in }L^1\loc(\Om).
\]
\end{lemma}\smallskip
\begin{proof}~
Let  $u\in SBV^2\loc(\Om,\C)$ such that $(u^+)^2=(u^-)^2$ on its jump set $J_u$. By the chain rule for $BV$-functions (see~\cite[Theorem~3.96]{AFP}), for every $f\in C^1(\C)$ with $\|Df(z)\|_\oo\le C(1+|z|)$ for some $C\ge0$, $f\circ u\in SBV\loc(\Om)$ and
\[
D[f\circ u]=Df(u) D_au +[f(u^+)-f(u^-)]\otimes \nu_u\h^1\restr J_u.
\]
Setting $v:=u^2$ and applying this formula with $f(z)=z^2$, we get 
\[
Dv=2u D_au +[(u^+)^2-(u^-)^2]\otimes  \nu_u\h^1\restr J_u=2u D_au.
\]
In particular,  $v\in W^{1,1}\loc(\Om)$.
We first consider the domain, $\Om':=\{x\in\Om:u(x)\ne0\}$. With the abuse of notation  $(\pt_1 u, \pt_2u):=\nb_a u$, we have,
\be
\label{proof_lem_curlrot_1}
\curl_a u = \pt_1 u_2-\pt_2 u_1.
\ee
Now, we express the components of $v$  in terms of the components of $u$ and use the chain rule to compute their partial derivatives. Almost everywhere in $\Om'$, there hold, 
\[
v_1=u_1^2-u_2^2,\qquad v_2=2u_1u_2,\qquad |v|=u_1^2 + u_2^2,
\]
which yield,
\begin{align*}
\pt_1 v_2 - \pt_2 v_1&=2\lt[u_1\pt_1 u_2 + u_2\pt_1 u_1 - u_1\pt_2 u_1 + u_2\pt_2 u_2 \rt],\\
\pt_1 v_1 + \pt_2 v_2&=2\lt[u_1\pt_1 u_1 - u_2\pt_1 u_2 + u_1\pt_2 u_2 + u_2\pt_2 u_1 \rt].
\end{align*}
Using these identities in the formula of Definition~\ref{def_rot} and simplifying, we get, almost everywhere in $\Om'$ and with $(\rho_1,\rho_2):=\rot v$,
\begin{align*}
|v|\rho_1&=u_1^2\lt[u_1\pt_1 u_2 + u_2\pt_1 u_1 - u_1\pt_2 u_1 + u_2\pt_2 u_2 \rt]\\
&\qquad\qquad -u_1u_2\lt[u_1\pt_1 u_1 - u_2\pt_1 u_2 + u_1\pt_2 u_2 + u_2\pt_2 u_1 \rt]\\
&=|v|u_1(\pt_1u_2 -\pt_2u_1)\, \st{\eqref{proof_lem_curlrot_1}}=\, |v|u_1\curl_a u,\\ \\
|v|\rho_2&=u_1u_2\lt[u_1\pt_1 u_2 + u_2\pt_1 u_1 - u_1\pt_2 u_1 + u_2\pt_2 u_2 \rt]\\
&\qquad\qquad -u_2^2\lt[u_1\pt_1 u_1 - u_2\pt_1 u_2 + u_1\pt_2 u_2 + u_2\pt_2 u_1 \rt]\\
&=|v|u_2(\pt_1u_2 -\pt_2u_1)\, \st{\eqref{proof_lem_curlrot_1}}=\, |v|u_2\curl_a u.
\end{align*}
We obtain the desired identity  $\rot v = (\curl_a u)\ub$ almost everywhere in $\Om'$. \medskip

In the remaining region $\Om'':=\{x\in\Om:u(x)=0\}$, we have $\rot v=0$ by definition and $D_au=0$, hence 
\be\label{gradu=0onu=0}
\curl_a u=0\text{ almost everywhere in }\Om''
\ee 
(see below). Therefore, the identity $\rot v = (\curl_a \ub)\ub$ holds true almost everywhere in $\Om$ and since both sides of the identity define locally integrable functions, the identity holds in $L^1\loc(\Om)$. This ends the proof of the lemma. 

Justification of~\eqref{gradu=0onu=0}: Let $T\in C^1(\C)$ such that $T(z)=z$ for $|z|\le 1$, $T(z)=z/|z|$ for $|z|\ge2$. For $\dl>0$, let us set $u_\dl:=\dl T(\dl^{-1}u)$. On one hand, we have $u_\dl \to 0$ in $L^1\loc(\Om)$ as $\dl\dw0$ so $D u_\dl\to0$ in the sense of distributions. On the other hand, by the chain rule,
\begin{equation}
D u_\dl =DT(\dl^{-1}u)D_a u +\dl \lt[T(\dl^{-1}u^+) - T(\dl^{-1}u^-)\rt]\otimes\nu_u\h^1\restr J_u.
\label{nbu=0_on_u=0}
\end{equation}
As $\dl\dw0$, the second term in the right-hand side goes to $0$ in $\mc{D}'(\Om)$ and the first term converges to $ \un_{\Om''}D_a u$ in $L^1\loc(\Om)$ by the dominated convergence theorem. Identifying with the limit $D u_\dl\to0$ in $\mc{D}'(\Om)$, we get $D_au=0$ in $\Om''$.
\end{proof}

We now justify the equivalence of the energies $E'_\eps$ and $E''_\eps$ (recall definitions \eqref{AGeps} and \eqref{AGepsprime}). The problem is mostly to find a ``good'' square root for $v$, see \cite{DI03,Merlet06,IL2017}
for related results.

\begin{proposition}
\label{prop_equivalence_modeles}
Let $\eps>0$ and le $\Om$ be an open set with finite measure.  
\begin{enumerate}[(i)]
\item Let $u\in SBV(\Om,\C)$ be such that $E'_\eps(\ub\otimes\ub)<\oo$ and let $v:=u^2$. Then,  $v\in W^{1,1}(\Om,\C)$,  $\rot v=0$ and
\be\label{prop_equivalence_modeles_1}
E\e''(v)=E'\e(\ub\otimes\ub).
\ee
\item Conversely, if $v\in W^{1,1}(\Om,\C)$ is such that $E\e''(v)<\oo$, then 
there exists $u\in SBV^2(\Om,\C)$ with $v=u^2$, $\curl_a\ub=0$ and $u^++u^-=0$ $\h^1$-almost everywhere on $J_u$. In particular,~\eqref{prop_equivalence_modeles_1} holds.
\end{enumerate}\end{proposition}\smallskip

\begin{proof}~

\noindent
\textit{(i).} Let $v:=u^2$ with $u$ as in (i). The condition $E\e'(\ub\otimes\ub)<\oo$ implies $u\in SBV^2(\Om)$ and from the assumption on $\Om$ also $u\in  L^2(\Om)$. As a consequence $v\in L^1(\Om)$. By Lemma \ref{lem_curlrot} (replacing $SBV\loc^2$ by $SBV^2$ in the hypothesis), we have $v\in W^{1,1}(\Omega,\C)$ with $\nb v=2u\nb_au$ and $\rot v=0$. Moreover, using the convention $|\nb v|^2/|v|=0$ wherever $v=0$, we have,
\[
\int_\Om \dfrac{|\nb v|^2}{4|v|}\, = \int_\Om |\nb_au|^2. 
\]
Eventually, $(1-|v|)^2=(1-|u|^2)^2$ from which \eqref{prop_equivalence_modeles_1} follows.\medskip

\noindent
\textit{(ii).} Let $v$ be as in (ii).

\noindent
\textit{Step 1. Selection of a $BV$ square root $u$ of $v$.} We would like to define a square root of $v$ of the form $u_\vhi:=e^{i\vhi/2} \sigma(e^{-i\vhi} v)$ for some $\vhi\in\R$.  We follow a strategy similar to the one of~\cite{DI03}. To deal with the discontinuity of $\sigma$ through $(-\oo,0)$ and the fact that $z\mapsto |\sigma(z)|=\sqrt{|z|}$ is not smooth at 0, we need to smooth out $\sigma$. First, we introduce $\chi\in C^\oo(\R_+)$ such that $\chi(s)=s$ for $0\le s\le 1$, $0\le\chi'\le1$ and $\chi(s)=1$ for $s\ge 2$. Next, we consider for $0<\dl<\pi$ a $2\pi$-periodic odd function  $f_\dl \in C^\oo(\R,(-\pi/2,\pi/2))$ such that
\[ 
\begin{array}{rlrl}
&f_\dl(\te)=\te/2&\text{ for }&\te\in[0, \pi-\dl],\\ 
1/2&\ge f_\dl'(\te)\ge 0 &\text{ for }&\te\in[\pi -\dl, \pi-\dl/2],\\
0&\ge f_\dl'(\te)\ge -2\pi/\dl & \text{ for }&\te\in[  \pi-\dl/2,  \pi].
\end{array}
\]
For $z=r e^{i\te}\in \C$ and $0<\dl<\pi$, we define $\sigma^\dl(z):=\sqrt{r}f_\dl(\te)$ and for $0<\dl<\pi$, $\lambda\ge1$, 
\[
\sigma^{\dl,\lambda}(z):=\chi(\lambda|z|)  \sigma^{\dl}(z).
\]
Eventually, we set
\[
u_\vhi^{\dl,\lambda} := e^{i\vhi/2}\sigma^{\dl,\lambda}(e^{-i\vhi}v).
\]
By construction, $u_\vhi^{\dl,\lambda}\in W^{1,1}(\Om,\C)$ and $(u_\vhi^{\dl,\lambda})^2\to v$ pointwise, uniformly in $\vhi$ as $\dl\dw0$ and $\lambda\up\oo$. We also easily see that 
\[
\int_0^{2\pi}|\nb u_\vhi^{\dl,\lambda}|\, d\vhi \le C \dfrac{|\nb v|}{\sqrt{|v|}}\qquad\text{almost everywhere in }\Om,
\]
for some universal constant $C>0$. Integrating on $\Om$ and using Fubini, we have 
\[
\int_0^{2\pi} \lt(\int_\Om |\nb u_\vhi^{\dl,\lambda}|\,\rt)\, d\vhi \le C \int_\Om \dfrac{|\nb v|}{\sqrt{|v|}}\le |\Om|^{1/2}\lt(\int_\Om \dfrac{|\nb v|^2}{|v|}\,\rt)^{1/2}<\oo.
\]
Let $(\dl_n)\dw0$ and $(\lambda_n)\up\oo$, we deduce that there exists a sequence $(\vhi_n)\sub[0,2\pi)$,  such that $(u_{\vhi_n}^{\dl_n,\lambda_n})$ is bounded in $W^{1,1}(\Om,\C)$. Therefore, there exists $u\in BV(\Om,\C)$  such that up to extraction,  $u_{\vhi_n}^{\dl_n,\lambda_n}\to u$ almost everywhere and weakly star in $BV(\Om,\C)$. Passing to the limit in the relation $({u_{\vhi_n}^{\dl_n,\lambda_n}})^2\to v$ as $n\up\oo$, we get $u^2=v$, with $u\in BV(\Om,\C)$.\smallskip

\noindent
\textit{Step 2. Properties of $u$.} From the chain rule for $BV$ functions, we have 
\[
\nb v =2u\nb_a u + 2u \nb_c u + \lt[(u^+)^2-(u^-)^2\rt]\nu_u\h^1\restr{J_u}.
\]
By identification, we obtain $\nb v=2u\nb_au$, $\nb_c u=0$ and $(u^+)^2=(u^-)^2$ $\h^1$-almost everywhere on $J_u$. Using $|u|^2=|v|$, we have 
\[
\int_\Om |\nb_a u|^2\, = \int_\Om \dfrac{|\nb v|^2}{4|v|}\,<\oo,
\]
hence $u\in SBV^2(\Om,\C)$. Eventually, by Lemma~\ref{lem_curlrot}, we obtain $\curl_au=0$ and from the above identity, we get~\eqref{prop_equivalence_modeles_1}.
 \end{proof}


\section{Entropies and entropy production}\label{Sentrop_prod}

 We recall that we have fixed an arbitrary function $\chi_0\in C^{\infty}_c([0,+\infty),[0,1])$ with $\chi_0(1)=1$ and $\supp \chi_0\subset (1/2,2)$
 and that for us an entropy is a function $\Phib\in C^\infty (\C,\C^2)$ such that for $z\neq0$,
 \[
  \Phib(z)= \chi_0(z)\Phib\lt(\frac{z}{|z|}\rt)
 \]
and such that for every $u\in C^\infty(\Omega,\So)$ with $\curl u=0$, we have  $\nabla\cdot[\Phib(u)]=0$. Let us also recall that as an alternative, we could equivalently define  entropies by the condition (see e.g. \cite[Proposition 3]{IM12})
\[
\dfrac{d}{d\te}[\Phib(e^{i\te})]\in \C \binom{\cos\te}{\sin\te} \quad\text{ for every }\te\in\R.
\]
Notice that this condition differs from the one of~\cite{DKMO01} since 
we consider the constraint $\curl \ub=0$ instead of $\ncd \mb=0$ (they are equivalent up to a rotation of angle $\pi/2$). We now set up some definitions that will be used throughout the paper.
\begin{definition}~\label{def_entrop}

\begin{enumerate}[(i)]
\item We denote by $\ent$ the space of entropies and by $\entev$  the subspace of even entropies, namely the elements $\Phib\in\ent$ such that $\Phib(-z)=\Phib(z)$ for every $z\in\So$.\smallskip
 \item We denote by $C^\oo_a(\T,\C)$ the subspace formed by the $\pi$-antiperiodic  functions  (\textit{i.e. } such that $\lambda(\cd+\pi)=-\lambda$) in the space $C^\oo(\T,\C)$ of smooth $2\pi$-periodic complex-valued functions.\smallskip
 \item
We denote by $\As(\Om)$ the set of measurable functions $v:\Om\to\So$  such that (recall definition \eqref{def:entroprod}) $\mus_\Phib[v]$ is a Radon measure for every $\Phib\in\entev$.\smallskip
\item
We call zero-states the functions $v\in \As(\Om)$ such that $\mus_\Phib[v]= 0$ for every $\Phib\in\entev$. The subset of zero-states is denoted by $\As_0(\Om)$.
\end{enumerate}

\end{definition}

Recall from \eqref{def:Gamma} the definition of $\Ups[\lambda]$ which associate to each $\lambda\in C^\oo(\T,\C)$ an entropy.
The following lemma is adapted from ~\cite[Lemma 3]{DKMO01} (see also \cite[Proposition 4]{IM12}). 
\begin{lemma}~
\label{lem_Philambda}
\begin{enumerate}[(i)]
\item \label{lem_Philambda_pt1}The functions $\Ups[\lambda]$ are entropies. Moreover, the mapping $\Ups: C^\oo(\T,\C)\to\ent$ is one-to-one and onto.
\item   \label{lem_Philambda_pt2}$\Ups$ maps  $C^\oo_a(\T,\C)$ onto $\entev$.
 \end{enumerate}
 \end{lemma}
 \begin{proof} The first point is the counterpart of~\cite[Proposition 4]{IM12} with the changes $\ub\leftrightarrow \mb:=\ub^\perp$, $\curl u\leftrightarrow \ncd\mb$.  
 
 For the second point, if $\lambda\in C_a^\oo(\T,\C)$, we easily check from the formula that $\Ups[\lambda]$ is even, hence from the first point, $\Ups[\lambda]\in\entev$. 
 
 \noindent
 Conversely, if $\Phib\in\entev$, by~\eqref{lem_Philambda_pt1}, there exists $\lambda\in C^\oo(\T,\C)$ such that $\Phib=\Ups[\lambda]$. Writing $\Phib(e^{i(\te+\pi)})=\Phib(-e^{i\te})=\Phib(e^{i\te})$ and taking the dot product with $(-\sin\te,\cos\te)$, we obtain $-\lambda(\te+\pi)=\lambda(\te)$ so $\lambda\in C_a^\oo(\T,\C)$, that is  $\Phib\in\Ups[C^\oo_a(\T,\C)]$.
 \end{proof}

We now recall the decomposition of $D \Phib$ established in~\cite{DKMO01} and reformulate it in our setting of $\curl$-free and $\C^2$-valued entropies. The resulting formula~\eqref{id_muPsialpha} will provide a control of the entropy productions $\mus_{\Phib}[v]$  in $H^{-1}(\Om)$ and in the space of Radon measures at the limit as stated in~Proposition~\ref{prop_muPhiH-1} and in~\eqref{control_mu_Phi0}.
\begin{lemma}
\label{lem_DPhi}
Let $\Phib\in \ent$  and for $z\in\C$ let $\lt(D\Phib\rt)_{i,j}(z):=\lt(\dfrac{\pt \Phib_i}{\partial z_j}(z)\rt)_{i,j}\in \mathcal{M}_{2,2}(\C)$ denote the matrix representation of the differential of $\Phib$ at $z\in\C\sim\R^2$.
\begin{enumerate}[(i)]
\item 
There exist $\Psib\in C^\oo_c(\C,\C^2)$ and $\alpha\in C^\oo_c(\C,\C)$ both supported in $B_2\sm\ov B_{1/2}$ such that
\[
D \Phib(z) + 2\Psib(z)\otimes {\bf z} =\alpha(z)J\qquad \text{for every }z\in\C\text{ and with }\ J:=\lt(\begin{array}{cc} 0&1\\-1&0\end{array}\rt).
\] 
In particular, for $u\in W^{1,2}(\Om,\C)$, we have  (recall  definition \eqref{def:entropintro}),
\[
\mu_{\Phib}[u]= \Psib(u)\cd\nb(1-|u|^2) + \alpha(u)\curl u.
\]
\item If moreover $\Phib\in\entev$, 
there exist $\Psis,\alfs\in C^\oo_c(\C,\C^2)$  such that 
\begin{enumerate}[(a)]
\item $\supp\Psis,\ \supp\alfs \sub B_{\sqrt{2}}\sm\ov B_{1/\sqrt{2}}$,
\item  for every $z\in\C$, $\alfs(z)$ is collinear to $\bs \sigma(z)$ ,
\item for every $v\in W^{1,1}(\Om,\C)$,
\be\label{id_muPsialpha}
\mus_{\Phib}[v]= \Psis(v)\cd\nb(1-|v|) + \alfs(v)\cd\rot v.
\ee
\end{enumerate}
\end{enumerate}
\end{lemma}
\com{Le point (ii)(b) ne sert pas pour l'instant mais on pourrait y repenser si on voulait remplacer $\|\rot v\|^2_{H^{-1}}$ par autre chose dans l'énergie. Pour la compacit\'e, c'est $\alfs(v)\cd\rot v$ qui intervient.}
\begin{proof} The first point is just a transposition of~\cite[Lemmas 1 and 2]{DKMO01} with the changes $\ub\leftrightarrow \ub^\perp=:\mb$ and $\curl \ub \leftrightarrow \ncd \mb$. The statement about the supports of $\alpha$ and $\Psib$ is obvious from the explicit definitions of $\Psib$ and then $\alpha$ given there. We now assume $\Phib\in\entev$ and establish (ii).\smallskip

\noindent\emph{Step 1. Symmetrization.} Differentiating the identity $\Phib(-z)=\Phib(z)$ and using the first part of the lemma, we obtain 
\[
-D\Phib(-z)= -2\Psib(-z)\otimes {\bf z} -\alpha(-z)J = -2\Psib(z)\otimes {\bf z} + \alpha(z)J,
\] 
In view of these identities, we can substitute the mean value $\frac12\lt(\Psib(z)+\Psib(-z)\rt)$ for  $\Psib(z)$ and  the quantity $\frac12\lt(\alpha(z)-\alpha(-z)\rt)$ for $\alpha(z)$. The new functions still comply to the conclusions of point (i) and we now have
\be\label{prf_lem_DPhi_1}
\lt\{\begin{array}{rcl}\Psib(-z)&=&\Psib(z),\\\alpha(-z)&=&-\alpha(z),\end{array}\rt.
\ \text{for }z\in \C. 
\ee
From now on, we assume that $\Psib$ and $\alpha$ satisfy~\eqref{prf_lem_DPhi_1} and the properties stated in (i).\medskip

\noindent\emph{Step 2. Smoothing.}
Let $v\in W^{1,1}(\Om,\C)$. Let $(v_k)\sub C^\oo(\Om,\C)$ be a sequence of approximations of $v$  such that, as $k\up\oo$, $(\nb v_k,v_k)\to(\nb v,v)$ in $L^1(\Om)$ and pointwise at every Lebesgue point of the mapping  
\[x\in\Om\mapsto (\nb v(x),v(x))\in\C^3.\]
By definitions \eqref{def:rotcomp} \& \eqref{def:entroprod} of $\rot v$ and $\mus_{\Phib}$, the assumptions on the sequence $(v_k)$ yield,
\be\label{prf_lem_DPhi_0}
\begin{cases}\rot v_k\to\rot v,\\ \mus_{\Phib}[v_k]\to \mus_{\Phib}[v],\end{cases}
\ \text{  at every Lebesgue point of }(\nb v,v).
\ee

Let us fix $x$, Lebesgue point of $(\nb v,v)$ with $v(x)\ne0$, let us fix $k$ large enough so that $v_k(x)\ne0$ and let $\xi:=v_k(x)/|v_k(x)|$. By continuity, there exists $0<r<\dist(x,\R^2\sm\Om)$ (depending on $k$) such that  $\vb_k\cd \xib>0$ in $B_{r}(x)$. The function $v_k/\xi$ is smooth and takes values in $\{z\in \C:\Re z>0\}$ in $B_r(x)$ so $u_k:=\sigma(\xi)\sigma(v_k/\xi)$ is smooth in $B_{r}(x)$ and $(u_k)^2=v_k$. Using Step~1, we write in $B_r(x)$,
\[
\mus_{\Phib}[v_k]= \mu_{\Phib}[u_k] =  \Psib(u_k)\cd\nb(1-|u_k|^2) + \alpha(u_k)\curl u_k.
\]
Using Lemma~\ref{lem_curlrot} and the identities, $u_k^2=v_k$, $|v_k|=|u_k|^2$, we get 
\begin{align}
\mus_{\Phib}[v_k]&=\Psib(u_k)\cd\nb(1-|v_k|)+ \dfrac{\alpha(u_k)}{|v_k|} \ub_k\cd\rot v_k.
\label{prf_lem_DPhi_2}
\end{align}
Now, taking into account the symmetries~\eqref{prf_lem_DPhi_1}, we define for $z\in\C$,
\[
\Psis(z):=\Psib(\sigma(z))=\Psib(-\sigma(z)),\qquad \alfs(z):=\dfrac{\alpha(\sigma(z))}{|z|}{\bs \sigma}(z)=\dfrac{\alpha(-\sigma(z))}{|z|}(-\bs \sigma(z)).
\]
These mappings are smooth, supported in  $B_{\sqrt{2}}\sm B_{1/\sqrt{2}}$ and by construction, $\alfs(z)$ is collinear to $\bs \sigma(z)$. Moreover, by~\eqref{prf_lem_DPhi_2}, the identity~\eqref{id_muPsialpha} holds true in $B_r(x)$ with $v_k$ in place of $v$. \medskip

\noindent\emph{Step 3. Sending $k$ to $+\oo$ and concluding.}~

Writing~\eqref{id_muPsialpha} with $v=v_k$ at point $x$, sending $k$ to $+\oo$ and recalling,
\[(\nb v_k(x),v_k(x))\ \st{k\up\oo}\longto\ (\nb v(x),v(x)),\]
and~\eqref{prf_lem_DPhi_0}, we obtain~\eqref{id_muPsialpha} at point $x$. 
This holds at every Lebesgue point $x$ of $(\nb v,v)$ such that $v(x)\ne0$. At Lebesgue points with $v(x)=0$, both sides of the identity vanish (because $\chi_0,\Psis,\alfs= 0$ in the neighborhood of 0) so~\eqref{id_muPsialpha} holds almost everywhere in $\Om$. Eventually, since both sides define locally integrable functions, the identity holds in $L^1\loc(\Om)$.
\end{proof}

We conclude this section by proving that   if $v\in W^{1,1}(\Om,\So)$ and if $\mus_\Phib[v]=0$ for at least one even entropy $\Phib$ which satisfies a mild non-degeneracy condition,  then $\rot v=0$. For instance, this holds for $\Phi=\Ups[\lambda]$ with $\lambda:\T\to \C$ analytic, non-constant and $\pi$-antiperiodic, in particular for the trigonometric entropies $\Phib^{n}$ with $n$ odd, $n\ne\pm 1$. The converse property is also true, if $\rot v=0$ and $v\in W^{1,1}\loc(\Om,\So)$ then $v\in\As_0(\Om)$.


\begin{lemma}~
\label{lem_rot=0}
Let $v\in W^{1,1}\loc(\Om,\So)$ and  $\lambda_0\in C_a^\oo(\T,\C)$ be such that the set of zeros of $\lambda_0+\lambda_0''$ is at most countable. 

Then the three following properties are equivalent,
\begin{center}
(i)~$\mus_{\Ups[\lambda_0]}[v]=0$ in $L^1\loc(\Om)$,\qquad (ii)~$\rot v=0$ in $L^1\loc(\Om)$,\qquad(iii)~$v\in\As_0(\Om)$.
\end{center}
\end{lemma}~

\begin{proof}~
Let $v\in W^{1,1}\loc(\Om,\So)$  and  $\lambda_0\in C_a^\oo(\T,\C)$ as in the statement of the lemma.  The implication (iii)~$\Rightarrow$~(i) is obvious, we prove below the implications  (i)~$\Rightarrow$~(ii) and then (ii)~$\Rightarrow$~(iii). In both cases, we use a $BV\loc$ lifting of $v$.\medskip

\noindent
\emph{Step 1. Choice a good lifting of $v$ and expressions of $\mus_{\Phib}[v]$ and $\rot v$.}

\noindent
Looking at the proof of Proposition \ref{prop_equivalence_modeles} (or directly using ~\cite{DI03,Merlet06}), we see that there exists $\te\in BV\loc(\Om)$ such that $v=e^{2i\te}$.  
 By the chain rule for $BV$-functions, 
%
we have with obvious notation,
\[
D v=2ivD_a\te + 2ivD_c\te  + [e^{2i\te^+}-e^{2i\te^-}]\nu_\te \h^1\restr J_\te.
\] 
 Identifying, we have $D_c\te=0$, $\te^+-\te^-\in \pi\Z$ $\h^1$-almost everywhere on $J_\te$ and, denoting $(\pt_1\te,\pt_2\te):=\nb_a \te$, 
\begin{equation}\label{eq:parv}
\pt_j v = 2(-\sin(2\te) + i \cos(2\te))\pt_j\te,\qquad \text{for }j=1,2.
\end{equation}

Let $\Phib\in\entev$, and $\lambda\in C^\infty_a(\T,\C)$  be such that $\Phib=\Ups[\lambda]$. Using the chain rule and denoting $R[\te]:=\cos(\te)\pt_1\te + \sin(\te)\pt_2\te$, we compute, 
\be
\mus_{\Phib}[v]=\ncd\lt[\Ups[\lambda](e^{i\te})\rt]
\label{proof_lem_rot=0_1} =-\lt[(\lambda''+\lambda)(\te)\rt] \,R[\te].
\ee
Similarly, using again the notation  $(\rho_1,\rho_2):=\rot v$ and \eqref{eq:parv}, there holds,
\begin{align*}
2\rho_1&=\lt(\cos(2\te) + 1\rt)(\cos(2\te)\pt_1\te + \sin(2\te)\pt_2\te) \\
 &~\qquad\qquad \qquad\qquad\qquad\qquad-\sin(2\te) (-\sin(2\te)\pt_1\te + \cos(2\te)\pt_2\te)\\
& =\lt[\lt(\cos(2\te) + 1\rt)\cos(2\te) + \sin(2\te) \sin(2\te)  \rt]\pt_1\te \\
&~\qquad\qquad\qquad\qquad\qquad\qquad+\lt[\lt(\cos(2\te) + 1\rt)\sin(2\te) - \sin(2\te) \cos(2\te)\rt]\pt_2\te\\
 &=\lt(\cos(2\te) + 1\rt)\pt_1\te + \sin(2\te) \pt_2\te\\
 &=2\cos(\te)R[\te],
 \end{align*}
and,
 \begin{align*}
 2\rho_2&=\sin(2\te)\cos(2\te)\pt_1\te + \sin(2\te)\pt_2\te\\
 &~\qquad\qquad \qquad\qquad\qquad\qquad  + \lt( \cos(2\te)  -1\rt) (-\sin(2\te)\pt_1\te + \cos(2\te)\pt_2\te )\\
 & =\lt[\sin(2\te)\cos(2\te)   - \lt(\cos(2\te) - 1\rt)\sin(2\te)\rt]\pt_1\te \\
 &~\qquad\qquad\qquad\qquad\qquad\qquad+\lt[\sin^2(2\te)+\lt(\cos(2\te) - 1\rt)\cos(2\te)\rt]\pt_2\te\\
 &=\sin(2\te)\pt_1\te +\lt(1-\cos(2\te)\rt) \pt_2\te\\
 &=2\sin(\te)R[\te].
\end{align*}
Hence, we have the identity
\be\label{proof_lem_rot=0_2}
\rot v = R[\te]\binom{\cos\te}{\sin\te}.
\ee\smallskip

\noindent
\emph{Step 2. (i)$\Rightarrow$(ii).}  Let us assume that $\mu_{\Ups[\lambda_0]}[v]=0$. Let 
\[Z_0:=\lt\{\vhi\in(-\pi,\pi]:(\lambda_0''+\lambda_0)(\vhi)=0\rt\}.\] 
Using~\eqref{proof_lem_rot=0_1} with $\Phib=\Ups[\lambda_0]$ we obtain that $R[\te]=0$ almost everywhere in  $\Om\sm\te^{-1}\lt(Z_0\rt)$ and by~\eqref{proof_lem_rot=0_2}, $\rot v=0$ almost everywhere in $\Om\sm\te^{-1}\lt(Z_0\rt)$.

The complement of $\te^{-1}\lt(Z_0\rt)$ is the union of the sets $v^{-1}(e^{2i\vhi})$ for $\vhi\in Z_0$. By assumption, this family of sets is at most countable and on each one   $\nb v$ vanishes almost everywhere by a standard property of Sobolev functions (see the justification of~\eqref{gradu=0onu=0} where the case of $SBV$ functions is treated).  We conclude that  $\nb v=0$  in $\te^{-1}\lt(Z_0\rt)$, hence $\rot v=0$ almost everywhere in $\Om$ which is~(ii).\medskip

\noindent
\emph{Step 3. (ii)$\Rightarrow$(iii).}  We assume that $\rot v=0$ so that \eqref{proof_lem_rot=0_2} implies $R[\te]=0$. From~\eqref{proof_lem_rot=0_1}, we deduce that $\mus_{\Phib}[v]=0$ for every $\Phib\in\entev$, as required. We have established $(iii)\Rightarrow(i)\Rightarrow(ii)\Rightarrow(iii)$, hence the lemma.
\end{proof}

\begin{remark}\label{rem:noentrot}
 Notice that combining \eqref{proof_lem_rot=0_1} and \eqref{proof_lem_rot=0_2} we see that
 as opposed to the oriented setting where the curl is the entropy production associated to the identity  (see Properties \ref{properties_Phin_2}),
 there is no  $\Phib=\Ups[\lambda]\in \entev$ such that $\rot v=(\Re (\mus_{\Phib}[v]),\Im(\mus_{\Phib}[v]))$. Indeed, $\lambda$ should solve $-(\lambda''+\lambda)=e^{i\te}$ which does not have periodic solution.
\end{remark}


\section{The trigonometric entropies}\label{SPhin}
In this section we collect the main properties of the trigonometric entropies from Definition \ref{def_Phi_n}. We recall that with $\Ups$ defined in \eqref{def:Gamma}, for every $n\in \Z$, we have set $\Phib^n=\Ups[2i e_n]\in \ent$, where $e_n$ is the trigonometric monomial $\te\in\R\mapsto e^{in\te}\in\C$.
We start with some  immediate consequences of the definition and of Lemma \ref{lem_Philambda}.
 \begin{properties}~
  \begin{enumerate}[(i)]
  \item For $n\in\Z$, we have $\Phib^{-n}=\ov{\Phib^n}$.
 \item
 By point~(i) of Lemma~\ref{lem_Philambda} and density of the trigonometric polynomials in $C^\oo(\T,\C)$, $\Span\{\Phib^n\}_{n\in\Z}$ is dense in $\ent$. 
 \item We have $e_n\in C^\oo_a(\T,\C)$ if and only if $n$ is odd. In fact, $\Span \{e_{2m+1}\}_{m\in\Z}$ is dense in $C_a^\oo(\T,\C)$. From Lemma~\ref{lem_Philambda}~(ii), we see that $\Span\{\Phib^{2m+1}\}_{m\in\Z}$ is dense in $\entev$.
 \end{enumerate}
 \end{properties}
 We derive an expression for the $\Phib^n$'s that first easily leads to the Properties~\ref{properties_Phin_2} below and then is used to establish the ``wedge products" formulas of Lemma~\ref{lem_Phinwedges}.
 \begin{lemma}~
 \label{lem_formulaPhin}
 \begin{enumerate}[(i)]
 \item For $n\in \Z$ and $z\in\So$, we have
 \be\label{closePhin}
 \Phib^n(z)=(n-1)z^{n+1} \binom1{-i} + (n+1) z^{n-1} \binom1i. 
 \ee
 \item For $m\in \Z$ and $w\in\So$, we have
 \be\label{closePhin2}
 \Phib^{2m+1}(\sigma(w))=2\lt[m w^{m+1} \binom1{-i} + (m+1) w^m \binom1i\rt]. 
 \ee
 \end{enumerate}
 \end{lemma}
 \begin{proof} Let us establish the first point. We write $z=e^{i\te}$. Using $\frac{d}{d\te}e^{in\te}=ine^{in\te}$ and then the formulas  $2i\sin\te=e^{i\te}-e^{-i\te}$, $2\cos\te=e^{i\te}+e^{-i\te}$, we get
 \begin{align*}
 \Phib^n(z)&= e^{in\te} \lt[2i \binom{-\sin\te}{\cos\te}+2n\binom{\cos\te}{\sin\te}\rt]=e^{in\te}\lt[\binom{-e^{i\te}+e^{-i\te}}{ie^{i\te}+ie^{-i\te}}+n\binom{e^{i\te}+e^{-i\te}}{-ie^{i\te}+ie^{-i\te}}\rt]\\
 &= e^{in\te}\binom{(n-1)e^{i\te}+(n+1)e^{-i\te}}{-i(n-1)e^{i\te}+i(n+1)e^{-i\te}} =  e^{in\te}\left[(n-1)e^{i\te} \binom1{-i} +(n+1)e^{-i\te} \binom1i\rt]. 
 \end{align*}
Substituting back $z=e^{i\te}$, $z^n=e^{in\te}$, we obtain~\eqref{closePhin}.  

Identity~\eqref{closePhin2} is obtained by substituting  $n=2m+1$ and $z=\sigma(w)$ in the first one and then using $z^2=w$.
 \end{proof}

%
Let us stress some immediate consequences of formulas~\eqref{closePhin}\&\eqref{closePhin2}.
 \begin{properties}~
 \label{properties_Phin_2}
  \begin{enumerate}[(i)]
  \item
 For $n=0$, we have $\Phib^0(z)=2i \zb^\perp$, so $\mu_{\Phib^0}[u]=2i\curl u$ for $u\in W^{1,1}(\Om,\So)$.
 \item
For $n=\pm 1$, we have 
\[\Phib^{\pm1}(z)=\pm 2 \binom1{\pm i},
\]
and the pair $\Phib^{-1}$, $\Phib^{1}$ spans the space of constant vector fields $\C\to \C^2$.
\item For $n=2$ and $z\in \So$, we have  $\Phib^{\pm2}(z)=6\,\Sigma_2(z)\pm i\,6\,\Sigma_1(z)$
where for $z\in\So$,
\[
\Sigma_1(z):=(z_2(1- (2/3)z_2^2), z_1(1-(2/3)z_1^2)),\qquad \Sigma_2(z):=(2/3)(z_1^3,-z_2^3)
\]
define the Jin--Kohn entropies. 
  \item  
For $n\in\Z$ and $\te\in\R$,  
\be
\label{normPhin}
\lt\|\Phib^n(e^{i\te})\rt\|_{\ell^2(\C^2)}=2\sqrt{n^2+1}\ \quad\text{and}\ \quad \lt\|\frac d{d\te}\lt[\Phib^{n}(e^{i\te})\rt]\rt\|_{\ell^2(\C^2)}\!=2|n^2-1|.
\ee
\end{enumerate}
 \end{properties}
 
 \begin{lemma}
 \label{lem_Phinwedges}
Let $n\in \Z$ and $z,w\in\So$. There holds, 
\begin{align}
\label{Phinwedge_1}
 \Phib^n(z)\we\Phib^{-n}(w) &=2i\lt((n+1)^2 \lt(z\ov w\rt)^{n-1} - (n-1)^2 \lt(z\ov w\rt)^{n+1} \rt),\\
\label{Phinwedge_2}
  \Phib^n(z)\we\Phib^{-n}(z)&=2i \lt((n+1)^2-(n-1)^2\rt)=8in,
 \end{align}
 and
 \begin{multline}
\label{Phinwedge_3}
\lt( \Phib^n(z)-\Phib^n(w)\rt)\we\lt( \Phib^{-n}(z)-\Phib^{-n}(w)\rt)\\
=2i\lt[(n+1)^2 \lt|w^{n-1}-z^{n-1}\rt|^2 - (n-1)^2  \lt|w^{n+1}-z^{n+1}\rt|^2\rt].
 \end{multline}
 \end{lemma}
 \begin{proof} The first equality is a direct application of~\eqref{closePhin} and of the identities,
 \be\label{id_det11ii}
 \binom1i\we\binom1i= \binom1{-i}\we\binom1{-i}=0,\quad\binom1{-i}\we\binom1i= -\binom1i\we\binom1{-i}=2i.
  \ee
 We get the second identity by taking $w=z$ in~\eqref{Phinwedge_1}.  For the last one we use the bilinearity of the wedge product and the two firsts to obtain that the left-hand side of~\eqref{Phinwedge_3} is equal to:
 \begin{multline*}
 2i(n+1)^2\lt[2-(\ov z w)^{n-1}-(z\ov w)^{n-1}\rt] - 2i (n-1)^2 \lt[2-(\ov z w)^{n+1}-(z\ov w)^{n+1}\rt]\\
 =2i(n+1)^2 \lt|z^{n-1}- w^{n-1}\rt|^2- 2i(n-1)^2\lt|z^{n+1}- w^{n+1}\rt|^2.
 \end{multline*}
 We used $z,w\in\So$ for the last equality. 
 This proves~\eqref{Phinwedge_3} and the lemma.
 \end{proof}
The next result is the key for passing from weak convergence to strong convergence at the end of the proof of the compactness result. 
\begin{lemma}
\label{lem_dirac}
Let $\nu$ be a Borel probability measure on $\So$. If one of the following assumptions holds true then $\nu$  is a Dirac mass.
\begin{enumerate}[(i)]
\item For every trigonometric entropy $\Phib^n$ with $n=2m$ even,
\[
\int_{\So} \lt[\Phib^{n}\we\Phib^{-n}\rt] d\nu=\lt[\int_{\So} \Phib^{n}\, d\nu\rt]\we \lt[\int_{\So} \Phib^{-n}\, d\nu \rt].
\]
\item For every trigonometric entropy $\Phib^n$ with $n=2m+1$ odd,
\be\label{lem_dirac_0}
\int_{\So} \lt[\Phib^{n}\circ\sigma\we\Phib^{-n}\circ\sigma\rt] d\nu=\lt[\int_{\So} \Phib^{n}\circ\sigma\, d\nu\rt]\we\lt[\int_{\So} \Phib^{-n}\circ\sigma\, d\nu \rt].
\ee
\end{enumerate}
\end{lemma}

\begin{proof}

\noindent
\emph{Case of assumption~(ii).}   Let $m$ be a positive integer and set $n:= 2m+1$.   Integrating  identity~\eqref{Phinwedge_2} of Lemma~\ref{lem_Phinwedges} with respect to $\nu$ we obtain for the left-hand side of~\eqref{lem_dirac_0},
\be\label{lem_dirac_prf1}
\dfrac1{8i}\int_{\So}\lt[\Phib^n\circ\sigma\we\Phib^{-n}\circ\sigma\rt] d\nu= \dfrac14\lt[(n+1)^2-(n-1)^2\rt]=(m+1)^2-m^2.
\ee
Next, using \eqref{closePhin2} and integrating with respect to $\nu$, we get,
\[
\dfrac{1}{2}\int_{\So} \Phib^n\circ\sigma\, d\nu = m \lt[\int_{\So} w^{m+1} \, d\nu(w)\rt] \binom1{-i} + (m+1) \lt[\int_{\So} w^m \, d\nu(w)\rt] \binom1i.
\]
Denoting by
\[
c_k:=c_k(\nu)=\int_0^{2\pi} e^{-ik\te}\,d\nu(e^{i\te})=\int_{\So} w^{-k }\, d\nu(w),
\] 
the $k^{\text{th}}$ Fourier coefficient of the probability measure  $\nu$, the last identity writes as,
\[
\dfrac{1}{2}\int_{\So} \Phib^n\circ\sigma\, d\nu =m\,c_{-(m+1)} \binom1{-i}+ (m+1)c_{-m} \binom1i,
\]
Using $\Phib^{-n}=\ov{\Phib^n}$ and the identities~\eqref{id_det11ii}, we get, for the right-hand side of~\eqref{lem_dirac_0},
\be\label{lem_dirac_prf2}
\dfrac1{8i}\lt[\int \Phib^n\circ\sigma\, d\nu\rt]\we\lt[\int\Phib^{-n}\circ\sigma\, d\nu\rt]=
m^2 |c_{m+1}|^2 - (m+1)^2|c_m|^2,
\ee
where we used $c_{-k}=\ov{c_k}$ for $k\in\Z$ (because $\nu$ is real valued). 

By assumption the left-hand side of~\eqref{lem_dirac_prf1} and~\eqref{lem_dirac_prf2} are equal. We deduce the relations 
\[
m^2(1- |c_{m+1}|^2) =  (m+1)^2(1-|c_m|^2)\quad\text{ for }m\ge1.
\]
This leads by induction to $1-|c_m|^2=m^2(1-|c_1|^2)$ for $m\ge 1$. Since $\nu$ is a (finite) measure, the sequence $(c_k)$ is bounded and we must have  $|c_1|=|c_1(\nu)|=1$. We conclude  that the probability measure  $\nu$ is a Dirac mass (we can for instance compute the variance $\Var(\nu)=\int |w|^2\, d\nu(w) - \lt|\int w\, d\nu(w)\rt|^2=1-|c_1(\nu)|^2=0$).\medskip

\noindent
\emph{Case of assumption~(i).} Performing the same computations with $n=2m$ in place of $n=2m+1$ and $\Phib^{2m}$ in place of $\Phib^{2m+1}\circ\sigma$ leads to the identities, 
\[1-|c_{2m-1}(\nu)|^2=(2m-1)^2(1-|c_1(\nu)|^2)\quad\text{ for }m\ge 1.\]
We conclude again that $|c_1(\nu)|=1$ and then that  $\nu$ is a Dirac mass.
\end{proof}

\section{Compactness}\label{Scomp}
In this section we prove  Theorem \ref{theo:introcomp}. As explained in the introduction, we first use Lemma \ref{lem_DPhi} to prove that the energy controls the entropy production (see \eqref{eq:entropyprodintro}). We will actually prove a slightly stronger statement which gives a more explicit control in terms of $1-|v|$, $\nabla v$ and $\rot v$ of the entropy production.\\
Fix $\chi\in C^{\infty}_c( (0,+\infty),[0,1])$ with $\chi(1)=1$ and  $\chi\ge 1/2$ on $[1/\sqrt{2},\sqrt{2}]$. For $\lambda_0,\lambda_1\ge 0$ with $\max(\lambda_0,\lambda_1)\ge1$ we define 
\begin{equation}\label{defQR}\left.\begin{array}{rl}
Q(v)&:=\|\chi(|v|)\,(1-|v|)\,\nb v\|_{1} + \lambda_0 \|\rot v\|_{1} +\lambda_1^{1/2}\|\chi(|v|)\,\nb v\|_{2} \,\|\rot v\|_{H^{-1}(\Om)},\smallskip\\
R(v)&:=\|\chi(|v|)\,(1-|v|)\|_{2}+ \lambda_1^{1/2}\|\rot v\|_{H^{-1}(\Om)}.
\end{array} \right.\end{equation}

\begin{proposition}
\label{prop_muPhiH-1}
Let $\Omega\subset \R^2$ be an  open set. Then, for every $\Phib\in \entev$ there exists $C=C(\Phib, \chi)>0$  such that 
 for every $v\in W^{1,1}(\Omega,\C)$ and every $\zeta\in C^1_c(\Omega,\C)$,
\begin{align}
\label{estim_H-1_muPhi}
\lt|\int_{\Omega}\mus_{\Phib}[v]\,\zeta\, \rt| &
\leq C\lt( Q(v)\|\zeta\|_\oo + R(v) \|\nb \zeta\|_{2}\rt).
\end{align}
\end{proposition}

\begin{proof}
 It is enough to prove the claim for either $(\lambda_0,\lambda_1)=(1,0)$ or $(\lambda_0,\lambda_1)=(0,1)$. Let $\Psis$, $\alfs$ denote the functions given by Lemma~\ref{lem_DPhi}~(ii) with the entropy $\Phib$. Applying~\eqref{id_muPsialpha} to $v\in W^{1,1}(\Omega,\C)$, we write
\begin{align*}
\ncd\lt[\Phib(\sigma(v))\rt] & = \ncd\lt[\Phib(\sigma(v)) - (1-|v|)\Psis(v)\rt] + \ncd\lt[(1-|v|)\Psis(v)\rt] \\
 &=-(1-|v|) D\Psis(v) D v + \alfs(v)\cd\rot v+ \ncd\lt[ (1-|v|)\Psis(v)\rt]\\
& =:\, f_1 + f_2 + f_3.
\end{align*}
Let $\zeta\in C^1_c(\Omega,\C)$. We estimate successively the terms $\lt|\int f_j\zeta\,\rt|$ for  $j=1,2,3$. For $f_2$, we consider two different bounds.  

\noindent
{\bf (1)} Since $D\Psis$ is bounded and supported in $B_{\sqrt{2}}\sm B_{1/\sqrt{2}}$, 
\[
\lt|\int_\Omega f_1\zeta\,\rt|\le C\lt\|\chi(|v|)(1-|v|)\nb v\rt\|_{1}\|\zeta\|_\oo.
\]
{\bf (2)} Next, since $\alfs$ is bounded, we have on the one hand the bound 
\[
\lt|\int_{\Omega} f_2\zeta\,\rt| \leq C \|\rot v\|_{1}  \|\zeta\|_\oo.
\]
On the other hand, by definition of the $\|\cd\|_{H^{-1}}$ norm, we also have 
\begin{align*}
\lt|\int_{\Omega} f_2\zeta\,\rt|&\le \|\rot v\|_{H^{-1}(\Omega)} \lt(\int_\Omega|\nb[\alfs(v)\zeta]|^2\,\rt)^{1/2}\\
&\le \|\rot v\|_{H^{-1}(\Omega)} \lt[\lt(\int_\Omega |D\alfs(v)Dv|^2\,\rt)^{1/2} \|\zeta\|_\oo  + \|\alfs(v)\|_\oo \|\nb \zeta\|_{2}\rt].
\end{align*}
As for  $\Psis$, the function $\alfs$ is bounded and supported in $B_{\sqrt{2}}\sm B_{1/\sqrt{2}}$ and thus
\[
\lt|\int_{\Omega} f_2\zeta\,\rt|\le C \|\rot v\|_{H^{-1}(\Om)} \Big[\lt\|\chi(|v|)\nb v\rt\|_{2} \|\zeta\|_\oo + \|\nb \zeta\|_{2}\Big].
\]
{\bf (3)} In the last term, we integrate by parts and use the Cauchy-Schwarz inequality to get 
\[
\lt|\int_{\Omega} f_3\zeta\,\rt|=\lt|\int_\Omega (1-|v|)\Psis(v)\cd\nb\zeta\rt|\le C \lt\| \chi(|v|)(1-|v|)\rt\|_{2}\|\nb \zeta\|_{2}.
\]
Summing the estimates for $j=1,2,3$, we conclude the proof of~\eqref{estim_H-1_muPhi}.
\end{proof}

We may now prove the main compactness result.
\begin{proposition}
\label{thm_comp}
Let $\Omega\subset \R^2$ be an  open set of finite area  and  $(v_k)\sub L^1(\Om,\C)$. 
\item Assume that:
\begin{enumerate}[(i)]
\item  $(|v_k|)$ converges to 1 in $L^{1}(\Om)$,
\item for every  $\Phi\in \entev$ and every $k\ge1$, there exists $C\ge 0$ and $\eta_k \dw 0$ such that   we can write
\[\lt|\int_\Omega \mus_{\Phib}[v_k] \zeta \rt|\le C\|\zeta\|_\infty +\eta_k \|\nabla \zeta\|_2, \quad \text{for every }\zeta\in C^1_c(\Omega,\C).
\]
\end{enumerate}
Then, up to extraction $v_k\to v$ in $L^1(\Om)$, for some  $v\in L^1(\Omega,\So)$. 
\end{proposition}

\begin{proof}

\noindent
\emph{Step 1. Convergence of $(v_k)$ as a sequence of Young measures.}

 Up to extraction, there exists a  positive Radon measure $\gamma$ over $\Omega\times \C$ such that for every $\vhi\in C_c(\Omega\times\C,\R)$, 
\be\label{conv_YM}
\int_\Omega \vhi(x,v_k(x))\,dx\ \st{k\up\oo}\longto\  \int_{\Om\times\C} \vhi\,d\gamma.
\ee
Moreover, $\gamma$ disintegrates as $\gamma=\nu_x\otimes \mathcal{L}^2$ where for almost every $x\in\Om$, $\nu_x$ is a positive Radon measure on $\C$. By assumption (i)  and the fundamental theorem on Young measures, see e.g. \cite[Theorem 3.1, (iii)\&(iv)]{MullerLecture}, $\nu_x$ is a probability measure supported in $\So$. Moreover, in order to prove the strong $L^1$ convergence of $(v_k)$, it is enough to establish
\be\label{nux_delta_mass}
\nu_x\text{  is a Dirac mass for almost every }x\in\Om.
\ee
This follows from assumption (i) again, \cite[Corollary 3.2]{MullerLecture} and Vitali convergence Theorem. Let us now establish~\eqref{nux_delta_mass}.

\medskip

\noindent
\emph{Step 2. $H^{-1}\loc$ compactness of the sequences of entropy productions.}
We observe that from assumption~(ii) of the theorem and a simple variant of a lemma by Murat~\cite{Murat81} (see also~\cite[Lemma~28]{Tartar79} and \cite[Lemma~6]{DKMO01}) applied to the (uniformly bounded) 
sequence of mappings $\Phib(\sigma(v_k))$,  
for any $\Phib\in\entev$,  the sequence $(\mus_{\Phib}[v_k])$  is relatively compact in $H^{-1}\loc(\Omega)$.
\medskip

\noindent
\emph{Step 3.  End of the proof. } 

Let $n=2m+1$ and $\Phib^n$ be the corresponding trigonometric entropy.
By Step 2, the sequences with terms $\mus_{\Phib^{\pm n}}[v_k]=\ncd\lt[\Phib^{\pm n}(\sigma(v_k))\rt]$ are relatively compact in $H^{-1}\loc(\Omega)$ and we can apply the div-curl lemma of Murat and Tartar~\cite{Murat78,Tartar79} to the pair of sequences, 
\[\lt(\Phib^{n}(\sigma(v_k))\rt),\qquad \lt(\lt[\Phib^{-n}(\sigma(v_k))\rt]^\perp\rt),
\] to get 
\[
\lim_{k\up\oo}\Phib^{n}(\sigma(v_k)) \we \Phib^{-n}(\sigma(v_k))=\lt[\lim_{k\up\oo} \Phib^{n}(\sigma(v_k))\rt]\we
\lt[\lim_{k\up\oo} \Phib^{-n}(\sigma(v_k))\rt],
\]
where the above limits are weak limits in $L^2\loc(\Om)$. Rewriting this identity with the limit Young measure, we obtain, for almost every $x\in\Om$,
\[ 
 \int_{\So} \lt[\Phib^{n}\circ\sigma\we\Phib^{-n}\circ\sigma\rt]d\nu_x = \lt[  \int_{\So} \Phib^{n}\circ\sigma\,d\nu_x\rt]\we\lt[\int_{\So} \Phib^{-n}\circ\sigma\,d\nu_x \rt].
\]
This identity holds true almost everywhere for any fixed $n$. Since we consider a countable number of entropies, it holds true for any odd $n$ on a set of full measure. For every $x$ in this set, we can apply Lemma~\ref{lem_dirac_0}~(ii) to the probability $\nu_x$ and deduce that it is a Dirac mass, that is~\eqref{nux_delta_mass}. This concludes the proof.
\end{proof}

As a direct consequence of Proposition \ref{prop_muPhiH-1} and Proposition \ref{thm_comp} we have the following compactness result.

\begin{corollary}
\label{coro_comp_1}
Let $\Om\sub \R^2$ be an open set of finite area and let $(v_k)\sub W^{1,1}(\Om,\C)$.
Assume that: 
\begin{enumerate}[(a)]
\item  $(|v_k|)$ converges to 1 in $L^{1}(\Om)$,
\item  $(R(v_k))$ converges to 0 and $(Q(v_k))$ is bounded.
\end{enumerate}
Then, up to extraction,
\begin{enumerate}[(i)] 
\item  $v_k\to v$ in $L^1(\Om)$, for some $v\in L^1(\Om,\So)$.
\item
The distributions $\mus_\Phib[v]$ are Radon measure and  there exists $C=C(\Phib)\ge 0$ such that 
\be\label{control_mu_Phi0}
|\mus_\Phib[v]|(\Om) \leq\, C\,  \liminf_{k\up\oo} Q(v_{k}).
\ee
\end{enumerate}
\end{corollary}

Recalling the definition \eqref{Eseps0} of $\Es_\eps$ we finally obtain the anticipated generalization of Theorem \ref{theo:introcomp}.
\begin{theorem}\label{thm:compmain}
 Assume that hypothesis \eqref{conditionsgw} holds. Let then $\Omega$ be an open set of finite measure and $\eps_k\dw 0$.
 If $(v_k)$ satisfies $\sup_k \Es_{\eps_k}(v_k)<+\infty$ then there exists $v\in L^1(\Omega,\So)\cap \As(\Om)$ such that up to extraction, $v_k\to v$ in $L^{1}$. Moreover, for every $\Phib\in\entev$, there exists $C=C(\Phib,\kappa)\ge 0$ such that 
 \[
  |\mus_\Phib[v]|(\Om) \leq\, C\,  \liminf_{k\up\oo} \Es_{\eps_k}(v_k).
 \]
\end{theorem}
\begin{proof}
 We first notice that by definition of $\Es_\eps$,
 and 
\begin{equation}\label{eq:firstboundEs}
 \int_\Om W(v)+\lambda_\eps^1\|\rot v\|^2_{H^{-1}(\Omega)}\le  \eps \Es_\eps(v).
\end{equation}
Using \eqref{conditionsgw} for $W$, this yields in particular
 \[
  \int_{\Omega} \min ((1-|v|)^2,|1-|v||) \le \eps \Es_\eps(v)/\kappa.
 \]
Since $\Omega$ is assumed to have finite area, we get first that if $ \Es_{\eps_k}(v_k)$ is bounded then $(|v_k|)$ converges to $1$ in $L^1(\Omega)$. Moreover, recalling the definition \eqref{defQR} of $R(v)$, \eqref{eq:firstboundEs} also gives
\begin{equation}\label{eq:boundR}
 R(v)\le \lt(\eps \Es_\eps(v)/\kappa\rt)^{1/2}.
\end{equation}
Next, using Young's inequality we find
\begin{multline*}
 \int_{\Om} g^{1/2}(v)W^{1/2}(v) |\nabla v|+\lambda_\eps^0\int_{\Omega} |\rot v| +\lt(\lambda_\eps^1 \|\rot v\|^2_{H^{-1}(\Omega)}\rt)^{1/2}\lt(\int_\Omega g(v) |\nabla v|^2\rt)^{1/2}\\
 \le C \Es_\eps(v).
\end{multline*}
Using again \eqref{conditionsgw} and recalling the definition \eqref{defQR} of $Q$ we get
\begin{equation}\label{eq:boundQ}
 Q(v)\le C \Es_\eps(v).
\end{equation}
From \eqref{eq:boundR} and \eqref{eq:boundQ} we conclude that if $\Es_{\eps_k}(v_k)$ is bounded then $Q(v_k)$ is also bounded and $R(v_k)$ goes to zero. We may therefore apply Corollary \ref{coro_comp_1} to conclude.
\end{proof}

\begin{remark}
 Let us point out that inserting \eqref{eq:boundR}\&\eqref{eq:boundQ} in  \eqref{estim_H-1_muPhi} from Proposition \ref{prop_muPhiH-1} we obtain \eqref{eq:entropyprodintro}, \textit{i.e.}
 \[
  \lt|\int_\Omega \mus_{\Phib}[v] \zeta\rt|\le C(\kappa)\lt(\Es\e(v)\|\zeta\|_\infty + \eps^{1/2}\Es\e^{1/2}(v)\|\nabla \zeta\|_2\rt).
 \]
\end{remark}


\section{Structure of zero-states}
\label{Sstructure_zero_states}
In this section we prove Theorem \ref{thm_0nrjw}. We first establish a regularity result for zero-states. We recall the notation
\[
 D_hf(x)=f(x+h)-f(x).
\]

\label{SS_lem_reg}
\begin{proposition}
\label{lem_reg}
Let $\Om$ be an open set and  $v\in\As_0(\Om)$ (see Definition \ref{def_entrop}). Then,
\begin{enumerate}[(i)]
\item \label{lem_reg_1p}
$v\in W^{1,3/2}\loc(\Om)$. 
\item  \label{lem_reg_rot=0}
$\rot v= 0$ in $\Om$.
\item \label{lem_reg_GL_estim} For every open set $\om\sub\sub\Om$, there exists $C=C(\om,\Om)\ge0$ such that for every $h\in\R^2$ with 
$\ds|h|\le \min\lt(1/2,\dist(\om,\R^2\sm\Om)/2\rt)$,
\be\label{lem_reg_estim}
\int_{\om} |D_h v|^2  \le C |h|^2 \ln(1/|h|).
\ee
\end{enumerate}
\end{proposition}
\begin{remark}\label{rem:regzero}~
\begin{enumerate}[(1)]
\item Let us stress once more that with point~\eqref{lem_reg_rot=0}, we recover a property which is not true for general configurations of $\As(\Om)$. 
\item Substituting $2^{(2-p)k}$ to $2^{k/2}$ in the proof of Step 3 below, we can establish that $v\in W^{1,p}\loc(\Om)$ for $1<p<2$. However, the constant degenerates as $p\up 2$ 
leading to a weaker version of~\eqref{lem_reg_estim} with  $ |h|^2 [\ln(1/|h|)]^2$ in place of $ |h|^2 \ln(1/|h|)$ (see Remark~\ref{rem:remregintro}). 
 \end{enumerate}
\end{remark}\smallskip

\begin{proof}[Proof of Proposition~\ref{lem_reg}]~

Since the result is local, it is enough to prove it for $\omega$ a ball (we will only use the fact that $\om$ is simply connected). Set $r:= \min\lt(1/2,\dist(\om,\R^2\sm\Om)/2\rt)$ and $\omega':= \omega+B_r\sub\sub \Omega$ (which is also simply connected). 
We fix in the proof a cut-off function $\zeta\in C^1_c(\Om,\R_+)$ supported in $\om$.

\noindent
We first introduce some notation.
\begin{enumerate}[(a)] 
\item 
For every $x\in \om$, we denote by $\dl_h(x)$ the unique element of $(-1,1]$ such that $v(x+h)=e^{i\pi\dl_h(x)}v(x)$.
\item For $k\ge0$, we define, 
\[
\om_k:=\lt\{x\in\om : \dfrac12<2^k|\dl_h(x)|\le1\rt\}.
\]
\item
For positive odd integers $n$ and $x\in \om$ we define (recall Definition \ref{def_Phi_n} of the trigonometric entropies $\Phib^n$),
\[
q_n(x):=\dfrac1{2i}\dfrac1{(n-1)^2(n+1)^2} \lt[D_h\lt[\Phib^n(\sigma(v))\rt](x)\rt]\we\lt[D_h\lt[\Phib^{-n}(\sigma(v))\rt](x)\rt].
\] 
\item 
For $k\ge0$ and $x\in \om$, we set,
\[
Q_k(x):=\sum_{m=2^k}^{2^{k+1}-1}q_{2m+1}(x),
\]
\item Eventually,  for $k\ge0$, we define the quantities,
\[
\mc{Q}_k:=\int_\Om Q_k\zeta^2.
\]
(In this formula, the quantity $Q_k(x)$ is arbitrary for $x\in\Om\sm\om$ (say  $Q_k(x)=0$). This makes no difference since $\supp\zeta\sub\om$.) 
\item In this proof we use the notation $A\les B$ to indicate that there is a universal constant $C\ge 0$ such that $A\le C B$.
\end{enumerate}
We split the proof into four steps. In the first two steps, we obtain bounds on the quantities $\mc{Q}_k$ from below and then from above.
In Step 3, we multiply these bounds by $2^{k/2}$ and sum over $k\ge0$.  Using  H\"older and sending $|h|$ to $0$ leads to $v\in W^{1,3/2}(\om)$. This establishes point~\eqref{lem_reg_1p} of the lemma.
By Lemma \ref{lem_rot=0} we then get $\rot v=0$, which is \eqref{lem_reg_rot=0}.  In the last step, we sum the bounds of Steps~1 and~2  over $k\in\{0,\dots,k_0\}$ where
$k_0:=\lceil\log_2(1/|h|)\rceil$. Using $v\in W^{1,1}\loc(\Om)$, we obtain \eqref{lem_reg_estim}.
\medskip

\noindent
\emph{Step 1. Lower bound.} 
We notice that since $\Phib^{\pm(2m+1)}\in\entev$, we have 
\[\Phib^{\pm(2m+1)}(\sigma(\xi e^{2i\vhi}))=\Phib^{\pm(2m+1)}(\sigma(\xi)e^{i\vhi})\qquad\text{for any }\xi\in\So,\ \vhi\in\R.\] In particular, for $x\in \om$ and $m\ge 1$,
\[
\Phib^{\pm(2m+1)}\lt(\sigma(v(x+h))\rt)= \Phib^{\pm(2m+1)}\lt(e^{i(\pi/2)\dl_h(x)}\sigma(v(x))\rt).
\]
With this in mind, we apply~\eqref{Phinwedge_3} of Lemma~\ref{lem_Phinwedges} with $n=2m+1$,  $m\ge1$ and $z=\sigma(v(x))$, $w=\sigma(v(x+h))=e^{i(\pi/2)\dl_h(x)}\sigma(v(x))$. We get, in $\om$,
\begin{align*}
q_{2m+1}&=\dfrac{\lt|e^{i\pi m\dl_h}-1\rt|^2}{4m^2} - \dfrac{\lt|e^{i\pi (m+1)\dl_h}-1\rt|^2}{4(m+1)^2}\\
&=\dfrac{\sin^2\lt(m(\pi/2)\dl_h\rt)}{m^2} - \dfrac{\sin^2\lt((m+1)(\pi/2)\dl_h\rt)}{(m+1)^2}.
 \end{align*}
Summing these identities for $m$ ranging over $\{2^k,\dots,2^{k+1}-1\}$, we get for $k\ge0$,
\[
Q_k
=\dfrac{\sin^2\lt((\pi/2)2^k\dl_h\rt)}{2^{2k}}-\dfrac{\sin^2\lt((\pi/2)2^{k+1}\dl_h\rt)}{2^{2(k+1)}} 
=\dfrac{\sin^4\lt((\pi/2)2^k\dl_h\rt)}{2^{2k}}.
\]
For  $x\in\om_k$, we have $1\ge 2^k|\dl_h(x)|>1/2$, hence 
\[
Q_k\ge \dfrac1{2^{2k}}\lt|2^k\dl_h\rt|^4 =2^{2k}|\dl_h|^4\ges |\dl_h|^2.
\]
Multiplying by $\zeta^2$ and integrating over $\Om$, we obtain,
\be\label{proof_lem_reg_a}
\mc{Q}_k \,\ges\,  \int_{\om_k}|\dl_h|^2 \zeta^2\,.
\ee
\smallskip

\noindent
\emph{Step 2. Upper bound.} Let $m\ge1$, since by assumption $\ncd[\Phib^{2m+1}(\sigma(v))]=0$ in the simply connected domain  $\om'\sub\sub\Om$, there exists $F^m\in \Lip(\om',\C)$ such that $\nb^\perp F^m=\Phib^{2m+1}(\sigma(v))$. Using an integration by parts and $\ncd[\Phib^{-(2m+1)}(\sigma(v))]=0$, we compute for $m\ge1$,
\begin{align}
\nonumber
\int_\Om q_{2m+1}\zeta^2 &= \dfrac1{4m^2(m+1)^2} \int_\Om \lt[D_h\nb^\perp F^m\rt]\we \lt[D_h\Phib^{-2m-1}(\sigma(v))\rt]\zeta^2\\
\nonumber
&= -\dfrac1{4m^2(m+1)^2} \int_\Om \lt[\nb D_hF^m\rt] \cd \lt[D_h\Phib^{-2m-1}(\sigma(v))\rt]\zeta^2\\
\label{proof_lem_reg_2}
&=\dfrac1{2m^2(m+1)^2} \int_\Om D_hF^m \lt[D_h\Phib^{-2m-1}(\sigma(v))\rt]\cd [\zeta\ \nb \zeta].
\end{align}
(To make the argument rigorous, we first regularize $\Phib^{-2m-1}(\sigma(v))$, perform the integration by parts and pass to the limit.)

\noindent
Let us recall the identities~\eqref{normPhin}. We have, for $\te\in\R$,
\begin{align*}
 \lt\|\Phib^{\pm(2m+1)}(e^{i\te})\rt\|_{\ell^2(\C^2)}&=2 \sqrt{(2m+1)^2+1}\le 4(m+1),\\
  \lt\|\frac d{d\te}\lt[\Phib^{\pm(2m+1)}(e^{i\te})\rt]\rt\|_{\ell^2(\C^2)}&=8m(m+1),
 \end{align*}
In particular, for $x\in \om'$,
\[
\lt\|\nb F^m(x)\rt\|_{\ell^2(\C^2)}=\lt\|\Phib^{2m+1}(\sigma(v(x)))\rt\|_{\ell^2(\C^2)}\les m.
\] 
We use these bounds in the form, 
\[
\lt\|D_hF^m(x)\rt\|_{\ell^2(\C^2)}\les m |h|,\qquad\quad
\lt\|D_h\Phib^{-2m-1}(\sigma(v))\rt\|_{\ell^2(\C^2)}\les m \min\lt(1,m|\dl_h|\rt).
\]
Using these inequalities to estimate the right-hand side of~\eqref{proof_lem_reg_2} we get,
\[
\int_\Om q_{2m+1}\zeta^2\, \les\dfrac{|h|}{m^2} \int_\Om \min\lt(1,m|\dl_h|\rt)\zeta |\nb \zeta|\,. 
\]
Summing over $m\in\{2^k,\dots,2^{k+1}-1\}$, we obtain, 
\be
\label{proof_lem_reg_b}
\mc{Q}_k=\int_\Om \,\sum_{m=2^k}^{2^{k+1}-1} q_{2m+1}\zeta^2 
\les \dfrac{|h|}{2^k} \int_\Om \min\lt(1,2^k|\dl_h|\rt)\zeta |\nb \zeta|\,. 
\ee

\noindent
\emph{Step 3. Proof of~\eqref{lem_reg_1p} and~\eqref{lem_reg_rot=0}.}

Multiplying~\eqref{proof_lem_reg_a} and~\eqref{proof_lem_reg_b} by $2^{k/2}$ and summing over $k\ge0$, we obtain 
\[
\sum_{k\ge0} \int_{\om_k} 2^{k/2}|\dl_h|^2 \zeta^2\,\les\, \sum_{k\ge0} 2^{k/2}\mc{Q}_k\,\les\,
|h| \int_\Om \lt[\sum_{k\ge0}\dfrac1{2^{k/2}} \min\lt(1,2^k|\dl_h|\rt)\rt]\zeta |\nb \zeta|\,.
\]
We use again $1/(2|\dl_h|)<2^k$ in $\om_k$ to estimate the left-hand side from below. We get,
\be
\label{proof_lem_reg_c}
\int_\Om |\dl_h|^{3/2} \zeta^2\,\les\,
 |h| \int_\Om \lt[\sum_{k\ge0}\dfrac1{2^{k/2}} \min\lt(1,2^k|\dl_h|\rt)\rt]\zeta |\nb \zeta|\,.
\ee

We now estimate the right-hand side. Let us fix $x\in \om$ with $\dl_h(x)\ne0$ and let us define $K\ge0$ as the integer such that $x\in\om_K$. We have 
\[
 \sum_{k\ge0}\dfrac1{2^{k/2}} \min\lt(1,2^k|\dl_h|\rt)\les \pi\sum_{k=0}^K 2^{k/2}|\dl_h| +\sum_{k>K} 2^{-k/2} \les  |\dl_h| 2^{K/2} +2^{-K/2} \les  |\dl_h|^{1/2}
\]
where in the last inequality we used that $2^{-(K+1)}\le |\dl_h|\le 2^{-K}$ in $\omega_K$.\\
\medskip

Plugging this estimate in~\eqref{proof_lem_reg_c} and dividing by $|h|^{3/2}$, we obtain 
\[
\int_\Om \lt(\dfrac{|\dl_h|}{|h|}\rt)^{3/2}\zeta^2\,\les \int_\Om\lt(\dfrac{|\dl_h|}{|h|}\rt)^{1/2}\zeta|\nb\zeta|\,.
\]
Applying  H\"older inequality with parameters $p=3$, $q=3/2$ to the functions $f=\lt(|\dl_h|/|h|\rt)^{1/2}\zeta^{2/3}$ and $g=\zeta^{1/3}|\nb\zeta|$ and simplifying, we get 
\[
\lt(\int_\Om \lt(\dfrac{|\dl_h|}{|h|}\rt)^{3/2}\zeta^2\,\rt)^{2/3}\,\les   \lt(\int_\Om\zeta^{1/2}|\nb\zeta|^{3/2}\,\rt)^{2/3}.
\]

Using that $|D_hv|\le \pi |\dl_h|$, and a standard characterization of Sobolev spaces, we deduce that $v\in W^{3/2}\loc(\Om,\So)$ with the estimate
\[
\lt(\int_\Om |\nb v|^{3/2}\zeta^2\,\rt)^{2/3}\,\les \lt(\int_\Om\zeta^{1/2}|\nb\zeta|^{3/2}\,\rt)^{2/3}.
\]
 Point~\eqref{lem_reg_1p} of the Lemma is established. As already explained at the beginning of this proof, since $v$ is a zero-state point~\eqref{lem_reg_rot=0}  then follows from Lemma~\ref{lem_rot=0}.  \medskip

\noindent
\emph{Step 4. Proof of ~\eqref{lem_reg_GL_estim}.}~

Let $k_0:=\lceil\log_2(1/|h|)\rceil$, that is $k_0$ is the integer defined by $1/2 < 2^{k_0}  |h| \le 1$. Since $|h|\le1/2$, we have $k_0\ge 1$.  Summing~\eqref{proof_lem_reg_a} and~\eqref{proof_lem_reg_b} over $k\in\{0,\dots,k_0-1\}$, we get 
\be\label{proof_lem_reg_d}
\sum_{k=0}^{k_0-1} \int_{\om_k} |\dl_h|^2 \zeta^2\,\les\, \sum_{k=0}^{k_0-1} \mc{Q}_k\,\les\,
|h| \int_\Om \lt[\sum_{k=0}^{k_0-1}\dfrac1{2^{k}} \min\lt(1,2^k|\dl_h|\rt)\rt]\zeta |\nb \zeta|\,.
\ee
We proceed as in Step~3 to estimate the right-hand side. Let us fix $x\in \om$ with $\dl_h(x)\ne0$ and let $K\ge0$ the integer such that $x\in\om_K$. 
We estimate the sum inside the brackets as follows
\begin{align*}
\sum_{k=0}^{k_0-1}\dfrac1{2^{k}} \min\lt(1,2^k|\dl_h|\rt) &\les \sum_{k=0}^{\min(k_0,K)-1} |\dl_h|+ \sum_{k=\min(k_0,K)}^{k_0}  \dfrac1{2^{k}}\\
&\les \,   k_0 |\dl_h| +  2^{-K} \les \,  ( k_0 + 1) |\dl_h| \les\,   |\dl_h| \ln(1/|h|),
\end{align*}
 In the last two  inequalities we used $2^{-K}\le 2|\dl_h(x)|$ for $x\in\om_K$  and $1\le k_0\le \log_2(1/|h|)$. Putting this inequality in~\eqref{proof_lem_reg_d}, we obtain
\[
\sum_{k=0}^{k_0-1} \int_{\om_k} |\dl_h|^2 \zeta^2\,\les\,
|h| \ln(1/|h|) \int_\Om |\dl_h|\,\zeta |\nb \zeta|\,.
\]
Using $|\dl_h| \le 2 |D_h v|$ and recalling that by Step~3, $v\in W^{1,3/2}\loc(\Om)\subset W^{1,1}\loc(\Om)$, we deduce,
\be\label{proof_lem_reg_e}
\sum_{k=0}^{k_0-1} \int_{\om_k} |\dl_h|^2 \zeta^2\,\le C |h|^2 \ln(1/|h|)\,,
\ee
where $C>0$ only depends on $\zeta$.\\
 Eventually, for $k\ge k_0$ and $x\in \om_k$ we have $|\dl_h(x)|\le 2^{-k_0}\le 2 |h|$, so that 
\[
\sum_{k\ge k_0} \int_{\om_k} |\dl_h|^2 \zeta^2\,\les  |h|^2 \int_\Om \zeta^2.
\]
Together with~\eqref{proof_lem_reg_e} (and $|h|\le1/2$) this leads to
\[
 \int_\Om |\dl_h|^2 \zeta^2 \,\le C |h|^2 \ln(1/|h|)\,.
\]
 Since $|D_h v|\le \pi  |\dl_h|$, we conclude the proof of~\eqref{lem_reg_estim}.  
\end{proof}
Before going further into the proof, let us state the analogue of Proposition~\ref{lem_reg} for classical zero-states.
\begin{proposition}
\label{lem_reg2}
Let $\Om$ be an open set and  $u\in L^1\loc(\Om,\So)$ be a classical zero-state, that is $\mu_{\Phi^n}[u]=0$ for $n\in\Z$. 
Then, $u \in W^{1,3/2}\loc(\Om)$ and for every open set $\om\sub\sub\Om$, there exists $C=C(\om,\Om)\ge0$ such that for every $h\in\R^2$ with 
$\ds|h|\le \min\lt(1/2,\dist(\om,\R^2\sm\Om)/2\rt)$,
\be\label{lem_reg_estim2}
\int_{\om} |D_h u|^2  \le C |h|^2 \ln(1/|h|).
\ee
\end{proposition}
\begin{proof}
Let us first notice that $v:=u^2\in\As_0(\Om)$ so that Proposition~\ref{lem_reg} provides $v\in W^{1,3/2}\loc(\Om)$ and a control on $|D_hv|$ of the form~\eqref{lem_reg_estim}. In order to conclude, we only have to establish that for any simply connected open set $\om\sub\sub\Om$ and  any $h\in\R^2$ such that $|h|\le \min\lt(1/2,\dist(\om,\R^2\sm\Om)/2\rt)$ there holds
\be\label{proof_lem_reg2_1}
\mc{L}^2(A_h\cap\om) \le C(\om) |h|^2\qquad\text{ where }\qquad A_h:=\lt\{x\in \Om : \ |D_h u(x)|>\sqrt{2}\rt\}.
\ee
Indeed, assuming~\eqref{proof_lem_reg2_1}, we have for every  $p>0$,  $\zeta\in C^1_c(\Om,\R_+)$ and  $h\in\R^2$ with $\ds|h|\le \min\lt(1/2,\dist(\supp \zeta,\R^2\sm\Om)/2\rt)$,
\[
\int_\Om |D_h u|^p\zeta^2\,= \int_{A_h} |D_h u|^p \zeta^2\,+\int_{\Om\sm A_h} |D_h u|^p\zeta^2\,\le C(\zeta) |h|^2 + 2^{-p/2}\int_\Om |D_h v|^p\zeta^2.
\]
Here we used~\eqref{proof_lem_reg2_1} and $\sqrt2|D_hu|\le|D_h v|$ in $\Om\sm A_h$. Choosing $p=3/2$ and using $v\in W^{1,3/2}\loc(\Om)$, we deduce 
\[
\int_\Om |D_h u|^{3/2}\zeta^2 \le C(\zeta) |h|^{3/2}\qquad\text{for }h\in \R^2\text{ with }|h|\text{ small enough},
\]
and we conclude that $u\in W^{1,3/2}\loc(\Om)$. Similarly, choosing $p=2$,~\eqref{lem_reg_estim} yields~\eqref{lem_reg_estim2}.\medskip

Let $\om\sub\sub\Om$ be an open ball, let $\zeta\in C^1_c(\om,\R_+)$ and let $h\in\R^2$ with \[|h|\le\dist(\om,\R^2\sm\Om)/2.\] To establish~\eqref{proof_lem_reg2_1}, we proceed as in the proof of Proposition~\ref{lem_reg} but using the pair of Jin-Kohn entropies $\Phib^{\pm2}$ (the first computations below correspond to the beginning of the proof of~\cite[Lemma~7]{LLP20} rewritten with our notation). Let us set
\[
q_2(x):=\dfrac1{18\,i} \lt[D_h\lt[\Phib^2(u)\rt](x)\rt]\we\lt[D_h\lt[\Phib^{-2}(u)\rt](x)\rt].
\] 
On the one hand, from~\eqref{Phinwedge_3} and elementary calculus\footnote{We use $\sin^2(\te)-\sin^2(3\te)/9=(8/9)(2+\cos(2\te))\sin^4(\te)\ge(8/9)\sin^4(\te)$.}, we have
\be\label{proof_lem_reg2_2}
\int_\Om q_2\zeta^2 =\int_\Om\lt(\lt|D_h u\rt|^2 - \dfrac{\lt|D_h [u^3]\rt|^2}{9}\rt)\zeta^2\, \ge\dfrac89 \int_\Om\lt|D_h u\rt|^4\zeta^2\, \ge\dfrac{32}9  \int_{A_h} \zeta^2\,.
\ee
On the other hand using $\mu_{\Phib^{\pm 2}}[u]=0$ we get, as in the proofs of Proposition~\ref{lem_reg} or of~\eqref{Besovintro},  
\[
\int_\Om q_2\zeta^2 \le C\,|h|\int_\Om |D_hu| \zeta\,
=C\,|h|\int_{A_h\cup [\Om\sm A_h]} \!|D_hu|\zeta\,
\le C'\lt(|h|\lt(\int_{A_h} \zeta^2\,\rt)^{1/2} +|h|^2\rt),
\]
with $C,C'\ge0$ depending on $\zeta$. In the last estimate, we used the Cauchy-Schwarz inequality and $\sqrt2|D_hu|\le|D_h v|$ in $\Om\sm A_h$ with $v\in W^{1,1}\loc(\Om)$.  With~\eqref{proof_lem_reg2_2}, we obtain $\int_{A_h} \zeta^2\le C(\zeta)|h|^2$ and then~\eqref{proof_lem_reg2_1} thanks to a covering argument.
\end{proof}

We now return to the proof of Theorem~\ref{thm_0nrjw} and show that~\eqref{lem_reg_estim} translates into a local control of the Ginzburg--Landau energy of any mollification of $v$. To this aim, we fix
$\rho\in C^\oo_c(\R^2,\R_+)$ with $\int \rho=1$ and for $\eta\in(0,1/2]$, we set $\rho_\eta:=\eta^{-2}\rho(\eta^{-1}\cd)$ and $v_\eta:=v*\rho_\eta$.
\begin{lemma}\label{lem_GL}
 Let $\Om$ be an open set and let $v\in L^1\loc(\Om,\So)$ such that the conclusion~\eqref{lem_reg_GL_estim} of Proposition~\ref{lem_reg} holds true. Then, for every open set $\om\sub\sub\Om$, there exists $C=C(\om,\Om)\ge 0$ such that for $\eta\in(0,1/2]$ with $\eta\le\dist(\om,\R^2\sm\Om)/4$,
\be\label{lem_GL_1}
\GL_\eta(v_\eta;\om):=\dfrac12\int_\om |\nb v_\eta|^2 + \dfrac1{4\eta^2} \int_\om (1-|v_\eta|^2)^2 \,\le\, C \ln(1/\eta). 
\ee
\end{lemma}
\begin{proof}
 For the reader's convenience we recall some classical computations (see for instance  \cite{dLI15,LP18}). We first compute for $x\in\om$
\begin{align*}
\nb v_\eta(x)&=\int_{B_\eta} v(x-y)\nb \rho_\eta(y)\, dy = \int_{B_\eta} [v(x-y)-v_\eta(x)]\nb \rho_\eta(y)\, dy\\
&=\int_{B_\eta\times B_\eta} [v(x-y)-v(x-z)]\nb \rho_\eta(y)\, \rho_\eta(z) \,dy\, dz,
\end{align*}
where we used $\int \nb \rho_\eta=0$. We deduce the estimate,
\[
|\nb v_\eta(x)|\le \frac{C}{\eta^3} \int_{B_\eta} \int_{B_\eta} |v(x-y)-v(x-z)|\rho_\eta(z)\, dz\, dy.
\]
 Squaring, integrating on $\om$, using Jensen inequality and Fubini, we obtain 
\[
\int_\om |\nb v_\eta|^2 \le \dfrac {C}{\eta^4} \int_{B_\eta} \int_{\om+B_\eta}  |v(y)-v(y-h)|^2\, dy\, dh.
\]
 Using~\eqref{lem_reg_estim} from Proposition~\ref{lem_reg}~\eqref{lem_reg_GL_estim}, we get 
\be\label{proof_thm_struct_1}
\int_\om |\nb v_\eta|^2 \le C \ln(1/\eta).
\ee
Next, since $v$ takes values in $\So$, we have for $x\in\om$,
\begin{align*}
0\le 1-|v_\eta(x)|^2 &= \int_{B_\eta\times B_\eta} \lt(1-v(x-y)\ov v(x-z)\rt)\rho_\eta(y)\rho_\eta(z)\, dy\, dz\\
&=\dfrac12 \int_{B_\eta\times B_\eta}\lt |v(x-y)-v(x-z)\rt|^2\rho_\eta(y)\rho_\eta(z)\, dy\, dz,
\end{align*}
where we used the trigonometric formulas  $1-\cos\te=2\sin^2(\te/2)=(1/2)|1-e^{i\te}|^2$. Integrating over $\om$, using Jensen inequality,  Fubini and~\eqref{lem_reg_estim} as above, we get
\[
\int_\om  (1-|v_\eta|^2)^2\le C   \eta^2 \ln(1/\eta).
\]
Together with~\eqref{proof_thm_struct_1} this leads to~\eqref{lem_GL_1}.
\end{proof}
We may now prove Theorem~\ref{thm_0nrjw}, which we restate in terms of $v$. We set $u^*(x):= (x_1+ix_2)/|x|$.
\begin{theorem}
\label{thm_0nrj}
For each $v\in\As_0(\Om)$, there holds $\rot v=0$ and  there exists a locally finite set $S\sub \Om$ such that:
\begin{enumerate}[(i)] 
\item $v$ is locally Lipschitz continuous in $\Om\sm S$,
\item for $x\in\Om\sm S$, $v=v(x)$ on the connected component of $[x+\R\bs\sigma(v(x))]\cap [\Om\sm S]$ which contains $x$.
\item for every $B=B_r(x^0)$ such that $2B:=B_{2r}(x^0)\sub \Om$ and $2B\cap S=\{x^0\}$,
\begin{enumerate}[(a)]
\item either $v(x)=(u^*)^2(x-x^0)$ in $B\sm\{x^0\}$,
\item or there exists $ \xib\in \So$ such that\smallskip
\begin{itemize}
\item[$\circ$] $v(x)=(u^*)^2(x-x^0)$ in $\ds\lt\{x\in B\sm\{x^0\} : (x-x^0)\cd\xib\ge0\rt\}$,\smallskip
\item[$\circ$]  $v$ is  Lipschitz continuous in $\ds\lt\{x\in B\sm\{x^0\}: (x-x^0)\cd\xib\le0\rt\}$.
\end{itemize}
\end{enumerate}
\end{enumerate}
\end{theorem}
\begin{remark}
We can deduce the structure of zero-states in the classical setting by proceeding as in the proof below using Proposition~\ref{lem_reg2}, Lemma~\ref{lem_GL} and (more) simple geometric arguments. This provides an alternative proof of the main results of~\cite{JOP02}. 
\end{remark}
\begin{proof}[Proof of Theorem~\ref{thm_0nrj} (Theorem~\ref{thm_0nrjw})] The fact that $\rot v=0$ is established in \eqref{lem_reg_rot=0} of Proposition \ref{lem_reg} so we only need to prove (i)--(iii).\\
\noindent\emph{Step 1. Identification of the singular set.}  By \eqref{lem_GL_1} of Lemma \ref{lem_GL} and \cite[Theorem 4.1]{AlicPon}, the Jacobians
$\curl (v_\eta \wedge \nabla v_\eta)$ locally weakly converge as $\eta \dw 0$ (for the flat norm) 
to a measure $\mu=2\pi \sum_{i} z_i  \delta_{x_i}$ with $z_i\in \Z$. Moreover, for every $\omega\sub \sub \Omega$, there exists $C,C'\ge0$ depending on $\om$ and $\Om$ such that 
\[
 |\mu|(\om)\le C \limsup_{\eta\dw0}\frac{\GL_\eta(v_\eta;\om)}{\ln(1/\eta)} \stackrel{\eqref{lem_GL_1}}{\le} C'.
\]
Therefore the sum is locally finite. From~\eqref{lem_reg_1p} of Proposition~\ref{lem_reg}, $v\in W^{1,3/2}\loc(\Om)$, hence $\nabla v_\eta$ converges strongly
in $L^{3/2}\loc(\Om)$ to $\nabla v$ and $v_\eta$ converges strongly in $L^6\loc(\Omega)$ to $v$ so that $v_\eta \wedge \nabla v_\eta$ converges to $v\wedge \nabla v$ in $L^1\loc(\Om)$. Hence, $\mu=\curl (v\wedge \nabla v)$ and we can set $S:=\supp \mu$.
\medskip

\noindent
\emph{Step 2. Local Lipschitz regularity of $v$ in $\Om\sm S$.} Let $B=B_r(x)$ be an open ball such that $2B\sub\sub\Om\sm S$. Since $v\in W^{1,3/2}(2B)$ and  $\mu= \curl (v\wedge \nabla v)=0$ on $2B$,
by \cite{Demengel,BMP}, we can find $\vhi\in W^{1,1}(2B)$ such that $v= e^{i\vhi}$. 
We now set $\te:=\vhi/2$ and  $u:=e^{i\te}$ so that  $u^2=v$ with $u\in W^{1,3/2}(2B,\So)$.  From point~\eqref{lem_reg_rot=0}  of Proposition~\ref{lem_reg}, we have $\rot v=0$ in $2B$ so that   Lemma~\ref{lem_curlrot} implies $\curl u=0$. 
By the chain rule, this translates into
\be\label{proof_thm_struct_2}
\ub\cd\nb\te=0\quad\text{almost everywhere in }2B.
\ee
For $\lambda \in \R$ we define the level sets
\[
E_\lambda^-\coloneqq\{x\in 2B:\te(x)<\lambda\} \quad \text{ and }\quad E^+_\lambda\coloneqq\{x\in 2B:\te(x)>\lambda\}.
\] 
By~\eqref{proof_thm_struct_2} and since $\te\in W^{1,1}(2B)$, for almost every $\lambda$ these sets have finite perimeter in $2B$ and the tangent on $\pt E_\lambda^\pm$ is collinear to $e^{i\lambda}$. Therefore $E_\lambda^\pm$ is the 
intersection with $2B$ of a locally finite number of stripes parallel to $e^{i\lambda}$. Observing
that for $\lambda_-\le\lambda_+$ the sets $E_{\lambda_-}^-$ and $E_{\lambda_+}^+$ do not intersect
in $2B$ and taking into account the orientation of the stripes, we claim that for almost every
$\lambda_-,\lambda_+$ with $\lambda_-<\lambda_+\le\lambda_-+\pi/2$
\be\label{proof_thm_struct_3}
\dist(E_{\lambda_-}^-\cap B, E_{\lambda_+}^+\cap B) \ge 2r\sin\lt(\dfrac{\lambda_+-\lambda_-}2\rt).
\ee
Indeed, let $x_{\lambda_-}\in E_{\lambda_-}^-\cap B$ and $x_{\lambda_+}\in E_{\lambda_+}^+\cap B$, where
$\lambda_-<\lambda_+\le\lambda_-+\pi/2$. Since $E_{\lambda_-}^-$ is made of stripes parallel to
$e^{i\lambda_-}$, the line segment $L_{\lambda_-}\coloneqq [x_{\lambda_-}+\R \bs e^{i\lambda_-}]\cap 2B$
is contained in~$E_{\lambda_-}^-$. Similarly,  $L_{\lambda_+}\coloneqq [x_{\lambda_+}+\R \bs
e^{i\lambda_+}]\cap 2B \subset E_{\lambda_+}^+$. Since $\lambda_+\not\equiv \lambda_- \pmod{\pi}$, the lines spanned by $L_{\lambda_-}$ and $L_{\lambda_+}$ intersect, and since $L_{\lambda_-}\cap L_{\lambda_+}\sub E_{\lambda_-}^-\cap E_{\lambda_+}^+=\void$ they do not intersect in $2B$.
Minimizing $\dist(L_{\lambda_-}\cap B,L_{\lambda_+}\cap B)$ under these constraints, the minimizer is given by
\begin{align*}
L_{\lambda_-} &= \left(x+e^{i\alpha}\left[-2r+\R \bs e^{-\frac{i\beta}{2}}\right]\right)\cap 2B\\
L_{\lambda_+} &= \left(x+e^{i\alpha}\left[-2r+\R \bs e^{\frac{i\beta}{2}}\right]\right)\cap 2B
\end{align*}
with $\lambda_{\pm} = \alpha\pm\frac{\beta}{2}$ (see Figure~\ref{Fig_step2}, where we assume $\alpha=0$).
Is is then easy to see that
\begin{equation}\label{eq:fig3}
|x_{\lambda_-}-x_{\lambda_+}|\ge\dist(L_{\lambda_-}\cap B, L_{\lambda_+}\cap B) \ge 2r\sin\lt(\dfrac{\lambda_+-\lambda_-}2\rt),
\end{equation}
hence the claim \eqref{proof_thm_struct_3} by the arbitrariness of $x_{\lambda_-}\in
E_{\lambda_-}^-\cap B$ and $x_{\lambda_+}\in E_{\lambda_+}^+\cap B$.
\begin{figure}[ht]
    \centering
		\centering
		\captionsetup{width=.55\textwidth}
		\includegraphics[width=0.35\textwidth]{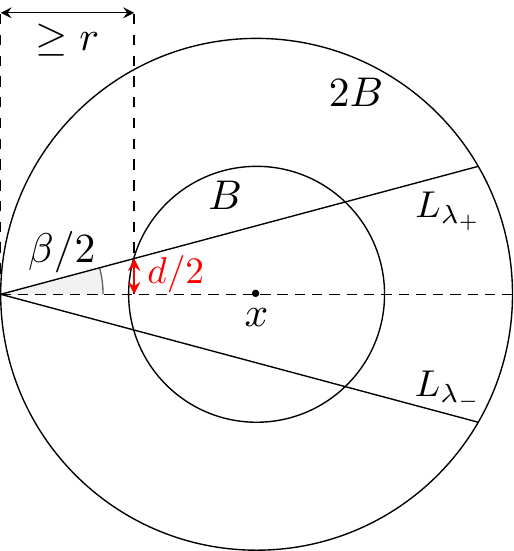}
		\caption{$d=\dist(L_{\lambda_-}\cap B,L_{\lambda_+}\cap B)$ as in \eqref{eq:fig3}}.
		\label{Fig_step2}

\end{figure}
From \eqref{proof_thm_struct_3} we deduce that $\te$ admits a Lipschitz continuous representative in  $ B$ with $\|\nb \te\|_{L^\oo(B)}\le 1/r$.  Therefore $v=e^{2i\te}$ admits a Lipschitz continuous representative in $\ov B$  with Lipschitz constant at most $2/r$. Eventually, returning to~\eqref{proof_thm_struct_2}, we see that for every $x\in B$ there holds
\[
v=v(x)\quad\text{on }[x+\R \bs e^{i\te(x)}]\cap B.
\]
Equivalently,
\[
v=v(x)\quad\text{on }[x+\R\bs \sigma(v(x))]\cap B.
\]
We have established points~(i) and~(ii) of the theorem.
\medskip

\noindent
\emph{Step 3. Structure of $v$ near the singularities.} Let $x^0\in S$ and $r>0$ such that $2B=B_{2r}(x^0)\sub \sub\Om$ with $2B\cap S=\{x^0\}$. By translation and scaling we may assume that $B= B_1$
and we set $B':=B\sm\{0\}$. We start by noting that since $\mu= \curl(v\wedge \nabla v)\neq 0$ in $2B$, $v$ must be discontinuous in $2B$.\medskip

\noindent
\emph{Step 3.a. Preliminaries.} For $x\in 2B'$ we denote by $L(x)$ the connected component of $[x+\R\bs \sigma(v(x))]$ in $2B'$ which contains $x$. $L(x)$ is an open segment in $\R^2$
with $L(x)\sub 2B'$. There are two cases:
\begin{enumerate}[(a)]
\item either $2B'\sm L(x)$ splits into two connected components. In this case the two endpoints of $L(x)$ lie on $\pt (2B)$, see Figure~\ref{Fig_case_a},
\item or $L(x)$ is a radius of the form $(0,2 x/|x|)$, see Figure~\ref{Fig_case_b}.
\end{enumerate}
For shortness, we denote $\xi(x):=\pm\sigma(v(x))$ where the sign is not important in case (a) but is chosen such that $\bs \xi(x)=x/|x|$ in case (b). With this notation, $L(x)$ is the connected component of $[x+\R\bs \xi(x)]$ in $2B'$ which contains $x$. 

\begin{figure}[ht]
    \centering
    \begin{minipage}{0.45\textwidth}
        \centering
		\captionsetup{width=.8\textwidth}
        \includegraphics[width=0.7\textwidth]{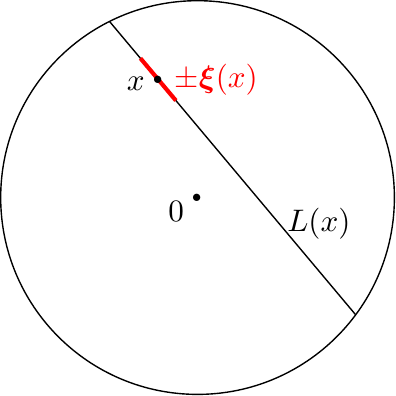}
        \caption{Case (a).}
		\label{Fig_case_a} 
    \end{minipage}\hfill
    \begin{minipage}{0.45\textwidth}
        \centering
		\captionsetup{width=.8\textwidth}
        \includegraphics[width=0.7\textwidth]{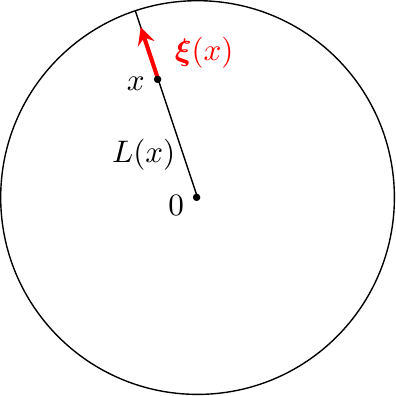}
        \caption{Case (b).}
		\label{Fig_case_b}
    \end{minipage}

	\vspace{4pt}
\end{figure}

From Step~2, $v$ is constant on $L(x)$ with value $(\xi)^2(x)$. This has the following consequences:
\begin{enumerate}[1.]
\item
 If $x,y\in 2 B'$ are such that $L(x)$ and $L(y)$ intersect then $v(x)=v(y)$ and $L(x)=L(y)$. 
 \item If $L(x)$ is of type (b), then $v=(u^*)^2$ on $L(x)$. 
\end{enumerate}

We use these properties repeatedly in the sequel with no further references.
\medskip

\noindent
\emph{Step 3.b. Construction of two radii of type (b).}
Let us establish the following.
\be\label{proof_thm_struct_4}
\text{There exist }x^1, x^2\in B'\text{ with }L(x^j)\text{ of type (b) for }j=1,2 \text{ and }\xi(x^2)\ne\pm \xi(x^1). 
\ee
Since $v$ is continuous in $B'$ and not continuous in $B$, there exist $\eps>0$
and two sequences $(x_k^1)$, $(x_k^2)$ converging to $0$ such that $|v(x^1_k)-v(x^2_k)|\ge\eps$ for every $k\ge1$.
Up to extraction, there exist $z^1, z^2\in\So$ with $|z^1-z^2|\ge\eps$ such that $v(x^{j}_k)\to z^{j}$ for $j\in \{1,2\}$. For the sequence of sets $(L(x^j_k))$, we have 
\[
L(x^j_k) \to L^j\quad \text{in Hausdorff distance, for }j\in \{1,2\}, 
\]
where $L^j$ is one of the segments $(0,\pm 2\sigma(\bs z^j))$ or the union of these two segments.
By continuity of $v$ in $2B'$, we have obtained two radii $L^1=(0,2\bs\xi^1)$, $L^2=(0,2\bs\xi^2)$ with $\xi^1,\xi^2\in\So$ such that $v=(\xi^j)^2=z^{j}$
on $L^j$ and with moreover $|(\xi^2)^2- (\xi^1)^2|\ge\eps$. In particular $L^2\ne\pm L^1$. Choosing $x^1\in L^1$ and $x^2\in L^2$, the sets $L^1=L(x^1)$, $L^2=L(x^2)$ are of type (b) 
with $x^1$, $x^2$ not collinear. This proves  claim~\eqref{proof_thm_struct_4}.\medskip

\noindent
\emph{Step 3.c Behavior of $v$ in convex sectors.}

Let $L^1=L(x^1)$, $L^2=L(x^2)$ be two radii with $x^1$, $x^2$ as in~\eqref{proof_thm_struct_4}. Let us denote by $P$ the open convex sector in $2B'$ delimited by $L^1$ and $L^2$. 
We claim that 
\be \label{proof_thm_struct_5}
v=(u^*)^2\quad\text{in }P.
\ee
Let us notice that since $L^2\ne\pm L^1$, the inner angle of $P$ at 0 is strictly smaller than $\pi$. Let $x\in P$ and let us consider the sequence $(x_k)=(2^{-k}x)$ and the sequence of segments $(L(x_k))$. Since the inner angle of $P$ is smaller than $\pi$, for $k$ large enough $x_k$ belongs to the convex hull of $L^1\cup L^2$, but $L(x_k)$ cannot intersect $L^1$ nor $L^2$, thus it is necessarily of type (b) (see Figure~\ref{Fig_case_3d}). Recalling that $x \in L(x_k)$, we deduce  $v(x)=(u^*)^2(x)$ on $L(x)=L(x_k)$. This proves~\eqref{proof_thm_struct_5}.
\medskip

\begin{figure}[ht]
    \centering
    \begin{minipage}{0.45\textwidth}
        \centering
		\captionsetup{width=.95\textwidth}
        \includegraphics[width=0.8\textwidth]{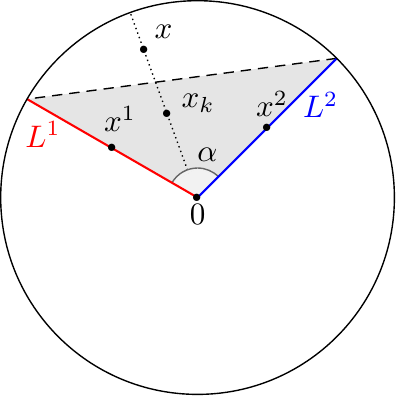}
		\caption{}
		\label{Fig_case_3c}
    \end{minipage}\hfill
    \begin{minipage}{0.45\textwidth}
        \centering
		\captionsetup{width=.95\textwidth}
        \includegraphics[width=0.8\textwidth]{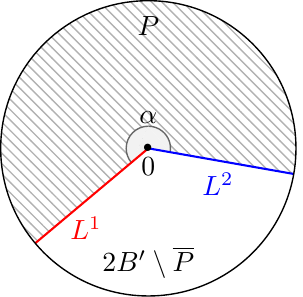}
		\caption{}
		\label{Fig_case_3d}
    \end{minipage}

\end{figure}

We now assume that $P$ is the open sector of $2B'$ with \emph{maximal} inner angle such that~\eqref{proof_thm_struct_5} holds true. We denote by $\alpha\in(0,2\pi]$ this angle and if $\alpha<2\pi$ (so that $P\ne 2B'$), $L^1$ and $L^2$ still denote its delimiting radii.\medskip

\noindent
\emph{Step 3.d. The case $\alpha> \pi$.} If $\pi<\alpha<2\pi$, then the complement of $\ov P$ in
$2B'$ is a convex sector with inner angle strictly smaller than $\pi$ and delimited by the two radii
$L^1$, $L^2$ (see Figure~\ref{Fig_case_3d}). Since by continuity $v=(u^*)^2$ on $L^1$ and $L^2$, we deduce from Step~3.c  that $v=(u^*)^2\text{ in }2B'\sm \ov P$. Hence $v=(u^*)^2$ in $2B'$ which contradicts the maximality of $P$. In conclusion, $\alpha=2\pi$ and $v=(u^*)^2$ in $2B'$. This corresponds to the case~(a) of the point~(iii) of the theorem.\medskip

\noindent
\emph{Step 3.e. The case $\alpha< \pi$.} Let us assume $\alpha<\pi$ and let us consider two sequences $(x^j_k)_{k\ge 1}\sub 2B'\sm\ov P$ for $j\in \{1,2\}$ such that $|x_k^1|=|x_k^2|=1$ for every $k$ and $x_k^j\to x^j\in L^j$ as $k\up\oo$. 

Let us assume by contradiction that $L(x_k^j)$ is of type (b)  (see
Figure~\ref{Fig_case_3e_type_b}) for $j=1$ or $j=2$. In this case, at least for $k$ large enough we
can apply Step~3.c. to the convex sector $Q_k$ generated by $x_k^j$ and $L^j$, we have $v=(u^*)^2$
in $P\cup Q_k$ and since $P$ and $Q_k$ have a common side this contradicts the maximality of $P$.
Therefore,  for $j=1$ and $j=2$, $L(x_k^j)$ is of type (a) for $k$ large enough (see
Figure~\ref{Fig_case_3e_type_a}). 

Next, by continuity of $v$ in $2B'$,  the segment $L(x_k^j)$ tends to be parallel to $L^j$ as $k\up\oo$ and since $\alpha<\pi$ we see that $L(x_k^1)\cap L(x_k^2)\ne\void$ 
for $k$ large enough and thus $L(x_k^1)= L(x_k^2)$. Passing to the limit we get $L^1=L^2$ which gives a contradiction.\medskip 

\begin{figure}[ht]
    \centering
    \begin{minipage}{0.45\textwidth}
        \centering
		\captionsetup{width=.95\textwidth}
        \includegraphics[width=0.8\textwidth]{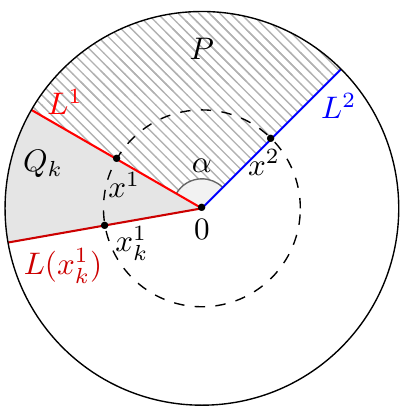}
        \caption{$L(x_k^1)$ (or $L(x_k^2)$) is of type (b). This contradicts the maximality of $P$.}
		\label{Fig_case_3e_type_b}
    \end{minipage}\hfill
    \begin{minipage}{0.45\textwidth}
        \centering
		\captionsetup{width=.95\textwidth}
        \includegraphics[width=0.8\textwidth]{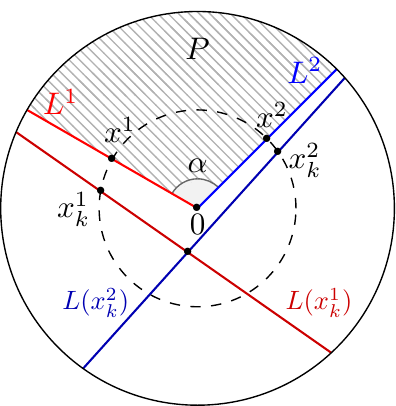}
        \caption{Both radii $L(x_k^1)$ and $L(x_k^2)$ are of type (a). This contradicts $L(x_k^1)\cap L(x_k^2)=\void$.}
		\label{Fig_case_3e_type_a}
    \end{minipage}

	\vspace{4pt}
\end{figure}
 

\noindent
\emph{Step 3.f. The case $\alpha=\pi$.} In this last case, there exists $\xi\in \So$ such that $P=\{x\in B':\bs\xi\cd x>0\}$. 
Let us establish that this corresponds to the case~(b) of the point~(iii) of the theorem. By~\eqref{proof_thm_struct_5}, we already 
know that $v=(u^*)^2$ in $P$ and we only have to prove that $v$ is  Lipschitz continuous in $\{x\in B': x\cd \bs\xi\le0\}$.

Let us set $Q:= \{x\in B: x\cd\bs\xi<0\}$. 
Let $x^1,x^2\in Q$ such that $v(x^1)\ne v(x^2)$ and $|v(x^1)-v(x^2)|<\sqrt{2}$. Then there exist $\theta_1,\theta_2\in\R$ such that $v(x^1)=e^{i\theta_1}$, $v(x^2)=e^{i\theta_2}$ and $|\theta_1-\theta_2|<\pi/2$. Since $v(x^1)\ne v(x^2)$, necessarily $L(x^1)\cap L(x^2)=\void$, and since $v(x^1)\neq \pm v(x^2)$ the lines spanned by $L(x^1)$ and $L(x^2)$ intersect outside $2B$.
Hence, the constraints on $L(x^1)$ and $L(x^2)$ are the same as those on $L_{\lambda_-}$ and $L_{\lambda_+}$ in Step~2, with $\theta_1$, $\theta_2$ in place of $\lambda_-$, $\lambda_+$.  Recalling that $\dist(L(x^1)\cap B,L(x^2)\cap B)$ is minimized in the situation of Figure~\ref{Fig_step2}, we find
\[
|x^1-x^2|\ge \dist(L(x^1)\cap B,L(x^2)\cap B) \ge 2\sin(|\theta_1-\theta_2|).
\]
This implies that $v$ is indeed Lipschitz continuous on $Q$. Eventually,  by continuity of $v$ in $B'$ we have that $v$ is Lipschitz continuous in  $\{x\in B': x\cd \bs\xi\le0\}$.
\end{proof}
\noindent
{\bf Acknowledgments.}
M. Goldman was partially supported by the ANR grant SHAPO.
B. Merlet and M. Pegon are partially supported by the Labex CEMPI (ANR-11-
LABX-0007-01).
S. Serfaty is supported by NSF grant DMS-2000205 and by the Simons Foundation through the Simons Investigator program.
\bibliographystyle{alpha}
\bibliography{curlfreeRP1.bib}


 \end{document}